\documentclass[preprint,10pt]{elsarticle}
\usepackage[utf8]{inputenc}
\usepackage{amsmath}
\usepackage{amsthm}
\usepackage{amssymb}
\usepackage{comment}
\usepackage{bbm}
\usepackage{comment}
\usepackage{float}
\usepackage{geometry}
\usepackage{graphicx}
\usepackage[usenames,dvipsnames]{xcolor}
\usepackage{hyperref}
\usepackage{natbib}
\usepackage{dsfont,url}
\usepackage{subcaption}
\usepackage{tikz}

\usetikzlibrary{positioning}
\usetikzlibrary{arrows}
\usetikzlibrary{arrows.meta}
\usetikzlibrary{calc}

\newtheorem{theorem}{Theorem}

\newtheorem{definition}[theorem]{Definition}

\newtheorem{lemma}[theorem]{Lemma}

\newtheorem{conjecture}[theorem]{Conjecture}

\def\qed{\hbox{${\vcenter{\vbox{		 %HOLLOW SQUARE
   \hrule height 0.4pt\hbox{\vrule width 0.5pt height 6pt
   \kern5pt\vrule width 0.5pt}\hrule height 0.4pt}}}$}}

\def\eps{\varepsilon}

\def\diag{\mathop{\rm diag}}

\newcommand{\ste}[1]{\textcolor{red}{#1}}

\begin{document}

\makeatletter
\def\ps@pprintTitle{%
  \let\@oddhead\@empty
  \let\@evenhead\@empty
  \let\@oddfoot\@empty
  \let\@evenfoot\@oddfoot
}
\makeatother

\begin{frontmatter}

\title{Global stability of SAIRS epidemic models}

\author[inst1,inst3]{Stefania Ottaviano}
\author[inst2]{Mattia Sensi}
\author[inst3]{Sara Sottile}

\affiliation[inst1]{organization={University of Trento, Dept.~of Civil, Environmental and Mechanical Engineering},%Department and Organization
            addressline={Via Mesiano, 77}, 
            city={Trento},
            postcode={38123}, 
            state={Italy}}
\affiliation[inst2]{organization={TU Delft, Network Architectures and Services Group},%Department and Organization
            addressline={Mekelweg 4}, 
            city={Delft},
            postcode={2628CD}, 
            state={The Netherlands}}
\affiliation[inst3]{organization={University of Trento, Dept. of Mathematics},%Department and Organization
            addressline={Via Sommarive 14}, 
            city={Povo - Trento},
            postcode={38123}, 
            state={Italy}}
%\journal{ }
\begin{abstract}
We study an SAIRS-type epidemic model with vaccination, where the role of asymptomatic and symptomatic infectious individuals are explicitly considered in the transmission patterns of the disease. 
We provide a global stability analysis for the model. We determine the value of the basic reproduction number $\mathcal{R}_0$ and prove that the disease-free equilibrium is globally asymptotically stable if $\mathcal{R}_0<1$ and unstable if $\mathcal{R}_0>1$, condition under which a positive endemic equilibrium exists.
We investigate the global stability of the endemic equilibrium for some variations of the original model under study and answer to an open problem proposed in Ansumali et al. \cite{ansumali2020modelling}.
In the case of the SAIRS model without vaccination, we prove the global asymptotic stability of the disease-free equilibrium also when $\mathcal{R}_0=1$. We provide a thorough numerical exploration of our model, to validate our analytical results.
 \end{abstract}
\begin{comment}

\begin{highlights}
\item We provide a global stability analysis for an SAIRS epidemic model with vaccination.
\item The role of both asymptomatic and symptomatic individuals is taken into account.
\item The threshold condition for the stability of the model is identified.
\item We analyze the behavior of the model when $\mathcal{R}_0 = 1$ for the model without vaccination.
\item We provide an in-depth numerical exploration of the model, to validate our analytical results.
\end{highlights}

\end{comment}
\begin{keyword}

 Susceptible--Asymptomatic infected--symptomatic Infected--Recovered--Susceptible \sep Vaccination \sep Basic Reproduction Number \sep Lyapunov functions \sep Global asymptotic stability \sep Geometric approach
 
\begin{comment}
\MSC[2010] 34A34 \sep 34D20 \sep 34D23 \sep 37N25 \sep 92D30

\end{comment}
\end{keyword}

\end{frontmatter}
\section{Introduction}

% \ste{The global asymptotic stability of the SAIR(S?) model with $\beta_A \neq \beta_I$, and $\delta_A \neq \delta_I$ is still open} %pointed out that the global asymptotic stability of the endemic equilibrium of the SAIR-type model with vital dynamics, when $\beta_A \neq \beta_I$ and $\delta_A \neq \delta_I$ was still an open problem

% The recent Covid-19 pandemic has shown the need to develop and study effective mathematical
%epidemic models, by incorporating new characteristic ingredients. 
The recent Covid-19 pandemic has demonstrated to what extend the study of mathematical models of infectious disease is crucial to provide particularly effective
tools to help policy-makers contain the spread of the disease.
 %Their evolution and persistence result from complex interactions between individual units,
%disease characteristics and control policies, which makes mathematical modelling a particularly effective
%tool for studying infectious diseases. 
Many large scale data-driven
simulations have been used to examine and forecast aspects of the current epidemic spreading \cite{aleta2020evaluation,gatto2020spread}, as well as in other past epidemics \cite{backer2016spatiotemporal,ferrari2008dynamics,tizzoni2012real}. However,
the study of theoretical effective epidemic models able to catch the salient transmission
patterns of an epidemic, but that are yet mathematical tractable,
offers essential insight to understand the qualitative behavior of the epidemic, and provides useful information for control policies.

A peculiar, yet crucial feature of the recent Covid-19 pandemic is that ``asymptomatic" individuals, despite showing no symptoms, are able to transmit the infection (see e.g., \cite{day2020covid,mizumoto2020estimating,oran2020prevalence,oran2021proportion}, where a considerable fraction of SARS-Cov-2 infections have been attributed to asymptomatic individuals). This is one of the main aspect that has allowed the virus to circulate widely in the population, since asymptomatic cases often remain unidentified, and presumably have more contacts than symptomatic cases, since lack of symptoms often implies a lack of quarantine. Hence, the contribution of the so called ``silent spreaders'' to the infection transmission dynamics should be considered in mathematical epidemic models \cite{robinson2013model}.

Models that incorporate an asymptomatic compartment already exist in literature \cite{debarre2007effect,kemper1978effects,stilianakis1998emergence}, but have not been analytically studied as thoroughly as more famous compartmental models.
%Subclinic al infections are typically modelled using two approaches . The first separates the infectious population into asymptomatic and symptomati c states immedia tely following the onset of infectiousnes s [4–6]. However, not all diseases follow such a straightfo rward path and, in particular, the relative duration s of
%the latent and incubation periods should be considered. For certain diseases, a more appropriate model may allow for a preclinical state where an asymptom atic infectious state precedes the onset of symptoms
In this work, we consider an SAIRS (Susceptible-Asymptomatic infected-symptomatic Infected-Recovered-Susceptible) model based on the one proposed in \cite[Sec. 2]{robinson2013model}, in which the authors provide only a local stability analysis. 
An SAIR-type model is studied in \cite{ansumali2020modelling} with application to SARS-CoV-2. 
After a global stability analysis of the model, the authors present a method to estimate the parameters. They apply the estimation method to Covid-related data from several countries, demonstrating that the predicted epidemic trajectories closely match actual data.
The global stability analysis in \cite{ansumali2020modelling} regards  only a simplified version of the model in \cite{robinson2013model}: first, recovered people do not lose their immunity; moreover, the infection rates of the asymptomatic and symptomatic individuals are equal, as well as their recovery rates, while in \cite{robinson2013model} these parameters are considered to be potentially different.  
 
Thus, the main scope of our work is to provide a global stability analysis of the model proposed in \cite{robinson2013model}, and for some variations thereof. In addition, we include in our model the possibility of vaccination. In the investigation of global stability, we answer an open problem left in \cite{ansumali2020modelling}. In particular, we study the global asymptotic stability (GAS) of the disease-free equilibrium (DFE) and provide
results related to the global asymptotic stability of the endemic equilibrium (EE) for many variations of the model, as we will explain in detail in Sec.~\ref{outline}.

The rigorous proof of global stability, especially for the positive endemic equilibrium, becomes a challenging mathematical problem for many disease models due to their complexity and high dimension \cite{shuai2013global}.
  
The classical, and most commonly used method for GAS analysis is provided by the Lyapunov stability theorem and LaSalle’s invariance principle. These approaches are successfully
applied, for example, to the SIR, SEIR and SIRS models (see, e.g. \cite{korobeinikov2006lyapunov,mena1992dynamic,shuai2013global}). 
Others techniques have appeared in literature, and were successfully applied to global stability arguments for various epidemic models. For example, the Li–Muldowney geometric approach \cite{li1996geometric,li2000dynamics} was used to determine the global asymptotic stability of the SEIR and SEIRS models \cite{li1999global,li1995global,van1999global}, of some epidemic models with bilinear incidence \cite{buonomo2008use}, as well as of SIR and SEIR epidemic models with information dependent vaccination \cite{buonomo2008global,buonomo2013modeling}. Applications of Li–Muldowney geometric approach can also be found in population
dynamics \cite{lu2012geometric}.

Unlike the more famous and studied epidemic models, much less attention has been paid to the SAIR(S)-type models.
Thus, we think that a deeper understanding of these kind of models is needed, and could prove to be very useful in the epidemiological field. Indeed, in various communicable diseases, such as influenza, cholera, shigella, Covid-19, an understanding of the infection transmission by asymptomatic individuals may be crucial in determining the overall pattern of the epidemic dynamics \cite{kemper1978effects, nelson2009cholera}.

In our model, the total population $N$ is partitioned into four compartments, namely $S$, $A$, $I$, $R$, which represent the fraction of Susceptible, Asymptomatic infected, symptomatic Infected and Recovered individuals, respectively, such that $N=S+A+I+R$. 
The infection can be transmitted to a susceptible through a contact with either an asymptomatic infected individual, at rate $\beta_A$, or a symptomatic, at rate $\beta_I$. This aspect differentiates an SAIR-type model from the more used and studied SEIR-type model, where once infected a susceptible individual enters an intermediate stage called ``Exposed'' (E), but a contact between a person in state $E$ and one in state $S$ does not lead to an infection.

In our model instead, once infected, all
susceptible individuals enter an asymptomatic state, indicating in any case a
delay between infection and symptom onset. We include in the asymptomatic class both individuals who will never develop the symptoms and pre-symptomatic who will eventually become symptomatic. The pre-symptomatic phase seems to have a relevant role in the transmission: for example, in the case of Covid-19, empirical evidence shows that the serial interval tends to be shorter than the incubation period, suggesting that a relevant proportion of secondary transmission can occur prior to symptoms onset \cite{gatto2020spread}; the importance of the pre-symptomatic phase in the transmission is underlined also for other diseases, such as dengue \cite{wiwanitkit2010unusual}, and H1N1 influenza \cite{gu2011pandemic}.

From the asymptomatic compartment, an
individual can either progress to the class of symptomatic infectious $I$, at rate $\alpha$,
or recover without ever developing symptoms, at rate $\delta_A$. An infected individuals with symptoms can recover at a rate $\delta_I$.
We assume that the recovered individuals do not obtain a long-life immunity and can return to the susceptible state  after an average time $1/\gamma$. We also assume that a proportion $\nu$ of susceptible individuals receive a dose of vaccine which grants them a temporary immunity.
We do not add a compartment for the vaccinated individuals,
not distinguishing the vaccine-induced immunity from the natural one acquired after recovery from the virus.
Moreover, we consider the vital dynamics of the entire population and, for simplicity, we assume that the rate of births and deaths are the same, equal to $\mu$; we do not distinguish between natural deaths and disease related deaths.

\subsection{Outline and main results}\label{outline}
In Sec.~\ref{SAIRS_sec}, we present the system of equations for the SAIRS model with vaccination, providing its positive invariant set.
In Sec.~\ref{extinction_sec}, we determine the value of the basic reproduction number $\mathcal{R}_0$ and prove that if $\mathcal{R}_0 <1$, the DFE is GAS.

In Sec.~\ref{GAS_sec}, we discuss the uniform persistence of the disease, the existence and uniqueness of the endemic equilibrium, and we investigate its stability properties. In particular, first we provide the local asymptotic stability of the EE, then we investigate its global asymptotic stability for some variations of the original model under study. We start by considering the open problem left in \cite{ansumali2020modelling}, where the global stability of an SAIR model with vital dynamics is studied. The authors consider a disease which confers permanent immunity, meaning that the recovered individuals never return to the susceptible state. Moreover, they impose the restrictions $\beta_A =\beta_I$ and $\delta_A = \delta_I$, and leave the global stability of the endemic equilibrium when $\beta_A \neq \beta_I$ and $\delta_A \neq \delta_I$, as an open problem. Thus, in Sec. \ref{glob_sair}, we directly solve the open problem left in \cite{ansumali2020modelling}, by considering an SAIR model (i.e., $\gamma=0$), with $\beta_A \neq \beta_I$ and $\delta_A \neq \delta_I$, including in addition the possibility of vaccination. We consider the basic reproduction number $\mathcal{R}_0$ for this model and prove that if $ \mathcal{R}_0>1$ the EE is GAS.
In Sec.~\ref{glob_sairs_equal}, we study the GAS of the EE for an SAIRS model (i.e., $\gamma \neq 0$) with vaccination, with the restrictions $\beta_A =\beta_I$ and $\delta_A = \delta_I$, proving that if $\mathcal{R}_0>1$ the EE is GAS. In Sec. \ref{glob_SAIRS_noequal}, we investigate the global stability of the SAIRS model, where $\beta_A \neq \beta_I$ and $\delta_A \neq \delta_I$, {i.e., the model proposed in \cite{robinson2013model}, with in addition the possibility of vaccination}. 
{In this case, we use a geometric approach to global stability for nonlinear autonomous systems due to Lu and Lu \cite{lu2017geometric}, that generalizes the criteria developed by Li and Muldowney \cite{li1996geometric,li2000dynamics}}. We prove that if $\mathcal{R}_0>1$ and $\beta_A < \delta_I$, the EE is GAS.

In Sec.~\ref{no_vax}, we are able to prove the GAS of the DFE also in the case $\mathcal{R}_0=1$, assuming that no vaccination campaign is in place.
In Sec.~\ref{num_analysis_sec}, we validate our analytical results via several numerical simulations and deeper explore the role of parameters. 

\section{The SAIRS model with vaccination}\label{SAIRS_sec}
We consider an extension of the SAIRS model presented in \cite{robinson2013model}. 
\begin{comment}The total population $N$ is divided into four compartments: susceptible $S$, asymptomatic infected $A$, symptomatic infected $I$ and recovered $R$.
We assume that all infections are acquired either from asymptomatic or symptomatic individuals \ste{at rate $\beta_A$ and $\beta_I$, respectively}, and once infected susceptible individuals enter in the asymptomatic compartment. Later, asymptomatic individuals may show symptoms after an average time $1/\alpha$ or recover after an average time $1/\delta_A$; it is also assumed that all infected individuals eventually recover and the immunity is not permanent. We assume that symptomatic individuals recover after an average time $1/\delta_I$ and that immunity is lost after an average time $1/\gamma$.
Moreover, we consider the vital dynamics of the entire population, assuming that deaths due to the disease are negligible.
We assume that the rate of birth and death are the same, and we indicate both with $\mu$; we do not distinguish between natural deaths and disease induced deaths.
\end{comment}
The system of ODEs which describes the model is given by
\begin{equation}\label{sairs_s}
\begin{split}
     \frac{d S(t)}{dt} &= \mu  - \bigg(\beta_A A(t) + \beta_I I(t)\bigg)S(t) -(\mu + \nu) S(t) +\gamma R(t),\\ 
     \frac{d A(t)}{dt} &=  \bigg(\beta_A A(t) + \beta_I I(t)\bigg)S(t) -(\alpha + \delta_A +\mu) A(t), \\ 
     \frac{d I(t)}{dt} &= \alpha A(t) - (\delta_I + \mu)I(t), \\ 
     \frac{d R(t)}{dt} &=  \delta_A A(t) +\delta_I I(t) + \nu S(t) - (\gamma + \mu)R(t),
     \end{split}
\end{equation}
with initial condition  $(S(0), A(0), I(0), R(0))$ belonging to the set
\begin{equation}\label{gamma_inv}
    \bar \Gamma=\{ (S, A, I, R) \in \mathbb R_+^{4}| S+ A+ I+R = 1\},
\end{equation}
where $\mathbb R_+^{4}$ is the non-negative orthant of $\mathbb R^{4}$. The flow diagram for system (\ref{sairs_s}) is given in Figure \ref{fig:SAIRS}.
\begin{figure}[h!]
			\centering
		\begin{tikzpicture}
		\node[draw,white,circle,thick,minimum size=.75cm] (smin) at (-0.1,2) {};
		\node[draw,white,circle,thick,minimum size=.75cm] (splu) at (0.1,2) {};
		\node[draw,white,circle,thick,minimum size=.75cm] (rmin) at (-0.1,0) {};
		\node[draw,white,circle,thick,minimum size=.75cm] (rplu) at (0.1,0) {};
		\node[draw,circle,thick,minimum size=.75cm] (s) at (0,2) {$S$};
		\node[draw,circle,thick,minimum size=.75cm] (a) at (3.5,2) {$A$};
		\node[draw,circle,thick,minimum size=.75cm] (i) at (3.5,0) {$I$};
		\node[draw,circle,thick,minimum size=.75cm] (r) at (0,0) {$R$};
		\draw[-{Latex[length=2.mm, width=1.5mm]},thick] (s)--(a) node[above, midway]{$(\beta_A A +\beta_I I)S$};
		\draw[-{Latex[length=2.mm, width=1.5mm]},thick] (a)--(i) node[right, midway]{$\alpha A$};
		\draw[-{Latex[length=2.mm, width=1.5mm]},thick] (splu)--(rplu) node[right, midway]{$\nu S$};
		\draw[-{Latex[length=2.mm, width=1.5mm]},thick] (rmin)--(smin) node[left, midway]{$\gamma R$};
		\draw[-{Latex[length=2.mm, width=1.5mm]},thick] (a)--(r) node[right, midway]{$\quad\delta_A A$};
		\draw[-{Latex[length=2.mm, width=1.5mm]},thick] (i)--(r) node[above, midway]{$\delta_I I$};
 		\draw[{Latex[length=2.mm, width=1.5mm]}-,thick]
 		(s)--++(-1.5,0) node[above,midway]{$\mu$};
 	 	\draw[-{Latex[length=2.mm, width=1.5mm]},thick]
 	 	(s)--++(0,1.5) node[left,midway]{$\mu S$};
		\draw[-{Latex[length=2.mm, width=1.5mm]},thick] (a)--++(1,1) node[above,midway]{$\mu A \quad$};
		\draw[-{Latex[length=2.mm, width=1.5mm]},thick] (i)--++(1,-1) node[above,midway]{$\quad\mu R$};
		\draw[-{Latex[length=2.mm, width=1.5mm]},thick] (r)--++(-1,-1) node[above,midway]{$\mu I \quad$};
		\end{tikzpicture}
			\caption{Flow diagram for system (\ref{sairs_s}).}
			\label{fig:SAIRS}
\end{figure}
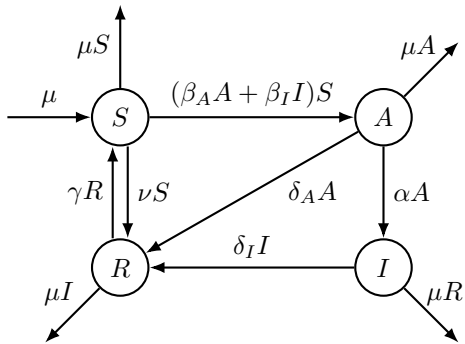\\
Assuming initial conditions in $\bar \Gamma$,  $S(t)+A(t)+I(t)+R(t)=1,$ for all $t\ge 0$; hence, system \eqref{sairs_s} is equivalent to the following three-dimensional dynamical system 
\begin{equation}\label{sairs3_s}
\begin{split}
    \frac{d S(t)}{dt} &= \mu -\bigg(\beta_A A(t) + \beta_I I(t)\bigg)S(t) -(\mu + \nu +\gamma) S(t) + \gamma(1-A(t)-I(t)), \\ 
      \frac{d A(t)}{dt} &=\bigg(\beta_A A(t) + \beta_I I(t)\bigg)S(t) -(\alpha + \delta_A +\mu) A(t),  \\
     \frac{d I(t)}{dt} &= \alpha A(t) - (\delta_I + \mu)I(t), 
\end{split}
\end{equation}
with initial condition  $(S(0), A(0), I(0))$ belonging to the set 
$$\Gamma=\{ (S, A, I) \in \mathbb R_+^{3}| S+ A+ I \leq 1\}.$$

System \eqref{sairs3_s} can be written in vector notation as

\begin{equation*}
    \frac{dx(t)}{dt}=f(x(t)),
\end{equation*}
where $x(t) = (S(t), A(t), I(t))$ and $f(x(t)) = (f_1(x(t)), f_2(x(t)), f_3(x(t)))$ is defined according to \eqref{sairs3_s}.

\begin{theorem}\label{invset}
  $\Gamma$ is positively invariant for system \eqref{sairs3_s}. That is, for all initial values $x(0) \in \Gamma$, the solution $x(t)$ 
  of \eqref{sairs3_s} will remain in $\Gamma$ for all $t >0$.
 \end{theorem}
 \begin{proof}
A compact set $C$ is invariant for the system $d x(t)/dt= f(x(t))$ if at each point $y \in \partial \Delta$ (the boundary of $C$), the vector $f(y)$ is tangent or pointing into the set \cite{yorke1967invariance}.

The boundary $\partial \Gamma$ consists of the following $4$ hyperplanes: 

\begin{align*}
H_1&=\{(S,A,I) \in \Gamma \; | \; S =0 \},  \qquad
H_2=\{((S,A,I) \in \Gamma \; | \; A =0 \},  \\
H_3&=\{(S,A,I) \in \Gamma \; | \; I=0 \},  \quad 
H_4=\{(S,A,I) \in \Gamma \; | \; S+A+I =1 \} \\
\end{align*}

whose respective outer normal vectors are:

\begin{align*}
\eta_1 =(-1, 0, 0), \qquad 
\eta_2 = (0,-1,0), \qquad  
\eta_3 = (0,0,-1), \qquad
\eta_4 = (1,1,1).
\end{align*}

Thus, let us consider a point $x \in \partial \Gamma$. To prove the statement, we distinguish among four cases.\\

\emph{Case 1:} $S=0$. Then, since $A+I \leq 1$

$$\langle f(x), \eta_1 \rangle = -\mu -\gamma(1-A-I) \leq 0.$$

\emph{Case 2:} $A=0$. Then, since $S\geq 0$, $I \geq 0$

$$\langle f(x), \eta_2 \rangle= -\beta_I I S\leq 0.$$

\emph{Case 3:} $I=0$. Then, since $A \geq 0$

$$\langle f(x), \eta_3 \rangle= -\alpha A \leq 0.$$

\emph{Case 4:} $S+A+I=1$. Then, since $S\geq 0$, $A \geq 0$, $I \geq 0$

$$\langle f(x), \eta_4 \rangle=-\nu S -\delta_A A-  \delta_I I \leq 0.$$

Thus, any solution that starts in $\partial \Gamma$ will remain inside $\Gamma$. 
\end{proof}

\section{Extinction}\label{extinction_sec}

In this section, we provide the value of the basic reproduction number, that is defined as the expected number of
secondary infections produced by an index case in a completely susceptible
population \cite{anderson1992infectious,diekmann2000mathematical}. This numerical value gives a measure of the potential for disease spread
within a population \cite{van2008further}.
Then, we investigate the stability properties of the disease-free equilibrium of the system \eqref{sairs3_s}, that is equal to

\begin{equation}\label{eq:DFE}
x_0= \left(S_0, A_0, I_0\right)=\left(\frac{\mu+\gamma}{\mu+\nu+\gamma}, 0,0\right).\end{equation}

\begin{lemma}\label{propR0}
The basic reproduction number $\mathcal{R}_0$ of \eqref{sairs3_s} is given by
\begin{equation}\label{R0}
\mathcal{R}_0  = \left ( \beta_A + \dfrac{\alpha \beta_I}{\delta_I + \mu}\right) \dfrac{\gamma + \mu}{(\alpha+\delta_A + \mu)(\nu + \gamma + \mu)}.
\end{equation}
\end{lemma}
\begin{proof}
Let us use the next generation matrix method \cite{van2002repnum} to find $\mathcal{R}_0$. System \eqref{sairs3_s} has 2 disease compartments, denoted by $A$ and $I$. We can write
\begin{align*}
\frac{d A(t)}{dt} &=  \mathcal{F}_1(S(t),A(t),I(t)) - \mathcal{V}_1(S(t),A(t),I(t)) , \\ \nonumber
     \frac{d I(t)}{dt} &= \mathcal{F}_2(S(t),A(t),I(t)) - \mathcal{V}_2(S(t),A(t),I(t)), \\ \nonumber
\end{align*}
where 
\begin{align*}
&\mathcal{F}_1(S(t),A(t),I(t)) = \bigg(\beta_A A(t) + \beta_I I(t)\bigg)S(t), \qquad  \mathcal{V}_1(S(t),A(t),I(t)) = (\alpha + \delta_A +\mu) A(t),\\ \nonumber 
&\mathcal{F}_2(S(t),A(t),I(t)) = 0, \qquad  \mathcal{V}_2(S(t),A(t),I(t)) = -\alpha A(t) + (\delta_I +\mu) I(t).
\end{align*}

Thus, we obtain
\begin{equation}
F = \left( \begin{matrix} \label{F}
\dfrac{\partial \mathcal{F}_1}{\partial A}(x_0) & \dfrac{\partial \mathcal{F}_1}{\partial I}(x_0)\\ \\
\dfrac{\partial \mathcal{F}_2}{\partial A}(x_0) & \dfrac{\partial \mathcal{F}_2}{\partial I}(x_0)
\end{matrix} \right) = \left( 
\begin{matrix}
\beta_A S_0 & \beta_I S_0 \\ 0 & 0
\end{matrix}\right),\qquad \text{where }S_0 = \dfrac{\gamma + \mu}{\gamma + \mu + \nu},
\end{equation}
\begin{equation}\label{V}
V = \left( \begin{matrix}
\dfrac{\partial \mathcal{V}_1}{\partial A}(x_0) & \dfrac{\partial \mathcal{V}_1}{\partial I}(x_0)\\ \\
\dfrac{\partial \mathcal{V}_2}{\partial A}(x_0) & \dfrac{\partial \mathcal{V}_2}{\partial I}(x_0)
\end{matrix} \right) = \left( 
\begin{matrix}
\alpha+\delta_A+\mu & 0 \\ -\alpha & \delta_I+\mu
\end{matrix}\right),
\end{equation}
from which
\begin{equation*}
 V^{-1} = \left(\begin{matrix}
\dfrac{1}{\alpha+\delta_A+\mu} & 0 \\ \\\dfrac{\alpha}{(\alpha+\delta_A+\mu)(\delta_I + \mu)} & \dfrac{1}{\delta_I + \mu}
\end{matrix}\right).
\end{equation*}
The next generation matrix is defined as $M := FV^{-1}$, that is 
\begin{equation*}
M = \left( \begin{matrix}
\dfrac{\beta_A S_0}{\alpha+\delta_A+\mu} + \dfrac{\alpha \beta_I S_0}{(\alpha+\delta_A+\mu)(\delta_I + \mu)}& \dfrac{\beta_I S_0}{\delta_I + \mu}\\ \\
0 & 0 
\end{matrix} \right).
\end{equation*}
The basic reproduction number $\mathcal{R}_0$ is defined as the spectral radius of $M$, denoted by $\rho(M)$. Thus, with a direct computation, we obtain \eqref{R0}. \end{proof}

In the following, we recall some results that we will use to prove the global asymptotic stability of the disease-free equilibrium $x_0$ of \eqref{sairs3_s}.

%\begin{definition}
%A square matrix is called a \emph{Metzler matrix} if all its off-diagonal components are non-negative, and it is called a \emph{Hurwitz matrix} if its eigenvalues all have negative real parts.
%\end{definition}

%\begin{lemma}\label{HM}
%Let $H$ a Hurwitz Metzler matrix and $D$ be a positive definite diagonal matrix. Then $DH$ is Metzler and Hurwitz.
%\end{lemma}
%\begin{proof}
%See \cite[Lemma 3]{yang2017heterogeneous}.
%\end{proof}

\begin{lemma}\label{spFV}
The matrix $(F-V)$  has a real spectrum. Moreover, if $\rho(FV^{-1}) <1$, all the eigenvalues of $(F-V)$ are negative. 
\end{lemma}

\begin{proof}
From \eqref{F} and \eqref{V}
\begin{equation}\label{FminusV}
 (F-V) = \left( 
\begin{matrix}
\beta_A S_0 -(\alpha+\delta_A+\mu) & \beta_I S_0  \\ \alpha & -(\delta_I+\mu)
\end{matrix}\right).
\end{equation}
%Let $P=\diag(\sqrt{\beta_I(r) S_0}, \sqrt{\alpha})$. Then, $P^1 \cdot(F-V)\cdot P$ is symmetric and hence its spectrum is real. Thus, since $F-V$ is similar to a matrix with real spectrum, its spectrum is also real.
 Since $(F-V)$ is a $2 \times 2$ matrix whose off-diagonal elements have the same sign, it is easy to see that its eigenvalues are real. Indeed, for a generic matrix $A=\left( 
\begin{matrix}
a & b  \\ c & d
\end{matrix}\right)$ with $\text{sign}(b)=\text{sign}(c)$, the eigenvalues can be easily shown to be real by explicitly computing them:
$$
\lambda_{1,2}=\dfrac{(a+d)\pm \sqrt{(a-d)^2+4bc}}{2},
$$
and noticing that the radicand is the sum of two non-negative values.
Now, if $\rho(FV^{-1})=\mathcal{R}_0 <1$ all eigenvalues of $(F-V)$ are negative as a consequence of \cite[Lemma 2]{van2008further}.
\end{proof}

\begin{theorem}\label{locDFE}
The disease-free equilibrium $x_0$ of \eqref{sairs3_s} is locally asymptotically stable if $\mathcal{R}_0 <1$, and unstable if $\mathcal{R}_0>1$.
\end{theorem}
\begin{proof}
See \cite[Theorem 1]{van2008further}.
\end{proof}

\begin{theorem}\label{glob_attr}
 The disease-free equilibrium $x_0$ of \eqref{sairs3_s} is globally asymptotically stable if $\mathcal{R}_0<1$.
\end{theorem}

\begin{proof}
Since $\Gamma$ is an invariant set for \eqref{sairs3_s} and in view of Theorem \ref{locDFE}, it is sufficient to show that for all $x(0) \in \Gamma$

\begin{equation*}
\lim_{t \to \infty} A(t)=0, \qquad      \lim_{t \to \infty} I(t)=0, \qquad \text{and} \qquad \lim_{t \to \infty} S(t)= S_0,
\end{equation*}
with $S_0$ as in \eqref{eq:DFE}. From the first equation of \eqref{sairs3_s} follows that
\begin{equation*}
\frac{dS(t)}{dt} \leq \mu+ \gamma  -(\mu+\nu+\gamma) S(t).   
\end{equation*}
It easy to see that $S_0$ is a global asymptotically stable equilibrium for the comparison equation

\begin{equation*}
\frac{dy(t)}{dt}= \mu+ \gamma  -(\mu+\nu+\gamma) y(t).   
\end{equation*}
Then, for any $\eps>0$, there exists $\bar t>0$, such that for all $t \geq \bar t$, it holds

\begin{equation}\label{S0e}
    S(t) \leq S_0+ \eps,
\end{equation}
hence 
\begin{equation}\label{limS}
\limsup_{t \to \infty} S(t) \leq S_0.
\end{equation}
Now, from \eqref{S0e} and second and third equation of \eqref{sairs3_s}, we have that for $t \geq \bar t$

\begin{align*}
    \frac{d A(t)}{dt} & \leq \bigg(\beta_A A(t) + \beta_I I(t)\bigg)(S_0 + \eps) -(\alpha + \delta_A +\mu) A(t),  \\
     \frac{d I(t)}{dt} &= \alpha A(t) - (\delta_I + \mu)I(t). \nonumber 
    \end{align*}
Let us now consider the comparison system
\begin{align*}
    \frac{d w_1(t)}{dt} & = \bigg(\beta_A w_1(t) + \beta_I w_2(t)\bigg)(S_0 + \eps) -(\alpha + \delta_A +\mu) w_1(t),  \\
     \frac{d w_2(t)}{dt} &= \alpha w_1(t) - (\delta_I + \mu)w_2(t), \qquad w_1( \bar t)=A( \bar t),\quad w_2( \bar t)=I( \bar t), \nonumber 
    \end{align*}
that we can rewrite as
\begin{equation*}
    \frac{d w(t)}{dt}= (F_{\eps}-V_{\eps}) w(t),
\end{equation*}
where $w(t)=(w_1(t),w_2(t))^T$ and $F_{\eps}-V_{\eps}$ is the matrix in \eqref{FminusV}, computed in $x_0(\eps)=(S_0+\eps, 0,0)$. Let us note that if $\mathcal{R}_0=\rho(FV^{-1}) <1$, we can choose a sufficient small $\eps >0$ such that $\rho(F_{\eps}V_{\eps}^{-1}) <1$. Then, by applying Lemma \ref{spFV} to $ (F_{\eps}-V_{\eps})$, we obtain that it has a real spectrum and all its eigenvalues are negative. It follows that $\lim_{t \to \infty} w(t)=0$ , whatever the initial conditions are (see, e.g., \cite{perko1991linear}), from which 
$$\lim_{t \to \infty} A(t)=0, \qquad \text{and} \qquad \lim_{t \to \infty} I(t)=0.$$ %for all $(A(0),I(0))$
Now, for any $\eps>0$ there exists $\bar t_1$ such that for any $t> \bar t_1$, $I(t)<\eps$ and $A(t)<\eps$. So, for $t>\bar t_1$ we have
$$
\frac{d S(t)}{dt}\ge  \mu -\eps(\beta_A + \beta_I)S(t) -(\mu + \nu +\gamma) S(t) + \gamma(1-2\eps). 
$$

It easy to see that $\frac{\mu+\gamma(1-2\eps)}{\eps(\beta_A + \beta_I)+(\mu + \nu +\gamma)}$ is a global asymptotically stable equilibrium for the comparison equation
\begin{equation*}
\frac{dy(t)}{dt}=\mu -\eps(\beta_A + \beta_I)y(t) -(\mu + \nu +\gamma) y(t) + \gamma(1-2\eps).
\end{equation*}
Thus, for any $\zeta >0$, there exists $\bar t_2 >0$ such that for all $t \geq \bar t_2$,
$$S(t) \geq \frac{\mu+\gamma(1-2\eps)}{\eps(\beta_A + \beta_I)+(\mu + \nu +\gamma)} - \zeta.$$
Then, for any $\eps>0$, we have
$$
\liminf_{t\to\infty}S(t)\ge\frac{\mu+\gamma(1-2\eps)}{\eps(\beta_A + \beta_I)+(\mu + \nu +\gamma)}.
$$
Letting $\eps$ go to $0$, we have $\liminf_{t\to\infty}S(t)\ge S_0$, that combined with \eqref{limS} gives us
$$\lim_{t \to \infty} S(t)=S_0.$$\end{proof}

\section{Global stability of the endemic equilibrium}\label{GAS_sec}

In this section, we discuss the uniform persistence of the disease, the existence and uniqueness of an endemic equilibrium, and we investigate its stability properties. \\

We say that the disease is \emph{endemic} if both the asymptomatic and infected fractions in the population remains above a certain positive level for a sufficiently large time. The notion of uniform persistent can be used to represent and analyze the endemic scenario \cite{li1999global}. In the following, with the notation $\mathring \Theta$, we indicate the interior of a set $\Theta$.

\begin{definition}\label{persistence_def}

System \eqref{sairs3_s} is said to be \emph{uniformly persistent} if there exists a constant $0<\eps<1$ such that any solution $x(t)=(S(t),A(t),I(t))$ with $x(0) \in \mathring \Gamma $ satisfies 

\begin{equation}\label{pers}
\min \{ \liminf_{t \to \infty} S(t), \quad \liminf_{t \to \infty} A(t), \quad \liminf_{t \to \infty} I(t)\} \geq \eps.
\end{equation}

\end{definition}

To address the uniform persistence of our system, we need the following result.

\begin{lemma}\label{uniqdfe}
The DFE $x_0$ is the unique equilibrium of \eqref{sairs3_s} on $\partial \Gamma$.
\end{lemma}
\begin{proof}
 Let us assume that $\bar x=(\bar S, \bar A, \bar I)$ is an equilibrium of \eqref{sairs3_s} on $\partial \Gamma$. Then, there are three possibilities:\\

\noindent\emph{Case 1:} $\bar S=0$. It follows from the second equation of \eqref{sairs3_s} that $\bar A=0$ and, consequently, from the third equation that $\bar I = 0$. Then, from the first equation of \eqref{sairs3_s} we have $\gamma(\bar A+\bar I)={\mu+\gamma}>0$, and a contradiction occurs.\\

\noindent\emph{Case 2:} $\bar A=0$. It follows from the third equation of \eqref{sairs3_s} that $\bar I=0$, and from the first that $\bar S=S_0$.\\

\noindent\emph{Case 3:} $\bar I=0$. Analogously to Case 2, we find that $\bar A=0$ and $\bar S=S_0$.\\

\noindent\emph{Case 4:} $\bar S+\bar A+ \bar I=1$. By summing the equations in \eqref{sairs3_s}, we have $\delta_A \bar A + \delta_I \bar I + \nu \bar S =0$, a contradiction.\\

By combining the above discussions the statement follows. 
\end{proof}

\begin{theorem}\label{uniform_per}
If $\mathcal{R}_0>1$, %( \mathcal{R}_0(r)=1 iff)
system \eqref{sairs3_s} is uniformly persistent and
there exists at least one endemic equilibrium in $\mathring{\Gamma}$.
\end{theorem}

\begin{proof}
By Lemma \ref{uniqdfe}, the largest invariant set on $\partial \Gamma$ is the singleton $\{ x_0\}$, which is isolated. If $\mathcal{R}_0>1$, we know from Theorem \ref{locDFE} that $x_0$ is unstable. Then, by using \cite[Thm 20]{freedman1994uniform}, and similar arguments in \cite[Prop. 3.3]{li1999global}, we can assert that the instability of $x_0$ implies the uniform persistence of \eqref{sairs3_s}. The uniform persistence and the positive invariance of the compact set $\Gamma$ imply the existence of an endemic equilibrium in $\mathring{\Gamma}$ (see, e.g., \cite[Thm 2.8.6]{bhatia2006dynamical} or \cite[Thm. 2.2]{shuai2013global}).
\end{proof}

%The uniform persistence, because of the boundedness of $\Gamma$ is equivalent to the
%existence of a compact set in which is absorbing for system

\begin{lemma}\label{ex_ee}
  There exists an endemic equilibrium $x^* = (S^*,A^*,I^*)$ in $\mathring{\Gamma}$ for system \eqref{sairs3_s} if and only if $\mathcal{R}_0>1$. Furthermore, this equilibrium is unique.
\end{lemma}
\begin{proof}

Let us consider \eqref{sairs3_s}; we equate the right hand sides to 0, and assume $A^*,I^*\neq 0$. From the third equation we obtain
\begin{equation}\label{Astar}
    A^* = \dfrac{\delta_I + \mu}{\alpha}I^*,
\end{equation}
and replacing it in the second equation
\begin{equation*}
    \left(\beta_A \dfrac{\delta_I + \mu}{\alpha} + \beta_I\right) I^* S^* - (\alpha+\delta_A + \mu)\dfrac{\delta_I + \mu}{\alpha} I^* = 0.
\end{equation*}
Since $I^*\neq 0$, it follows that
\begin{equation}\label{Sstar}
    S^* = \dfrac{(\alpha + \delta_A+\mu)(\delta_I + \mu)}{\beta_A(\delta_I + \mu) + \beta_I \alpha}.
\end{equation}
Let us substitute the expressions \eqref{Astar} and \eqref{Sstar} in the first equation, then we obtain

\begin{equation*}
    \mu - \left(\beta_A \dfrac{\delta_I + \mu}{\alpha} + \beta_I\right) \dfrac{(\alpha + \delta_A+\mu)(\delta_I + \mu)}{\beta_A(\delta_I + \mu) + \beta_I \alpha} I^* - (\mu + \nu + \gamma ) \dfrac{(\alpha + \delta_A+\mu)(\delta_I + \mu)}{\beta_A(\delta_I + \mu) + \beta_I \alpha} + \gamma \left(1- \dfrac{\delta_I + \mu}{\alpha} I^* - I^* \right) = 0,
\end{equation*}
which implies that
\begin{align}\label{Istar}
    I^* & = \dfrac{\mu - (\mu + \nu + \gamma) \dfrac{(\alpha + \delta_A+\mu)(\delta_I + \mu)}{\beta_A(\delta_I + \mu) + \beta_I \alpha} + \gamma}{\dfrac{1}{\alpha} (\beta_A (\delta_I + \mu) + \beta_I \alpha  )\dfrac{(\alpha + \delta_A+\mu)(\delta_I + \mu)}{\beta_A(\delta_I + \mu) + \beta_I \alpha}+ \gamma \dfrac{\delta_I + \mu}{\alpha} + \gamma}\nonumber\\
    & = \dfrac{(\mu+\gamma) (\beta_A(\delta_I + \mu) + \beta_I \alpha) - (\mu + \nu + \gamma )(\alpha + \delta_A + \mu) (\delta_I + \mu)}{ \dfrac{\beta_A(\delta_I + \mu) + \beta_I \alpha}{\alpha}\left( (\alpha + \delta_A + \mu+\gamma) (\delta_I + \mu) + \gamma \alpha\right)}  \nonumber\\
    & = \dfrac{(\delta_I + \mu) \left( (\mu + \gamma) \left(\beta_A + \beta_I \dfrac{\alpha}{\delta_I + \mu}\right) - (\mu + \nu + \gamma) (\alpha + \delta_A + \mu)\right)}{ \dfrac{\beta_A(\delta_I + \mu) + \beta_I \alpha}{\alpha}\left( (\alpha + \delta_A + \mu+\gamma) (\delta_I + \mu) + \gamma \alpha\right)} \nonumber \\
    & =\dfrac{(\delta_I + \mu) (\mu+\nu+\gamma)(\alpha+\delta_A+\mu)\left( \dfrac{(\mu + \gamma)}{(\mu + \nu + \gamma) (\alpha + \delta_A + \mu)} \left(\beta_A + \beta_I \dfrac{\alpha}{\delta_I + \mu}\right) - 1\right)}{ \dfrac{\beta_A(\delta_I + \mu) + \beta_I \alpha}{\alpha}\left( (\alpha + \delta_A + \mu+\gamma) (\delta_I + \mu) + \gamma \alpha\right)} \nonumber \\
    & = \dfrac{\alpha (\delta_I + \mu)(\mu+\nu+\gamma)(\alpha+\delta_A+\mu)} { (\beta_A(\delta_I + \mu) + \beta_I \alpha) \left( (\alpha + \delta_A + \mu+\gamma) (\delta_I + \mu) + \gamma \alpha\right)}(\mathcal{R}_0 -1).
\end{align}

The endemic equilibrium in $\mathring \Gamma$ exists if $A^* > 0 $ and $I^*> 0$. We obtain that $I^* >0$, and consequently $A^*>0$, if and 
only if $\mathcal{R}_0 - 1 >0 $. %thus the endemic equilibrium exists in the biologically relevant domain  only if $\mathcal{R}_0>1 $. 
\end{proof}

%\subsubsection{Local stability}

\begin{theorem}\label{loc_stab}
 The endemic equilibrium $x^* = (S^*,A^*,I^*)$ is locally asymptotically stable in $\mathring{\Gamma}$ for system  \eqref{sairs3_s} if $\mathcal{R}_0>1$.
\end{theorem}
\begin{proof}
Note that the expression of \eqref{Sstar} and \eqref{Istar} may be written as function of $\mathcal{R}_0$; using the expression found in \eqref{R0}, we obtain
\begin{align}
    S^* =& \dfrac{h_4}{\mathcal{R}_0},\label{Sstar2}\\
    I^* =& \dfrac{\alpha h_0 h_1 h_2 (\mathcal{R}_0 - 1) }{h_3(\beta_A h_2 + \beta_I \alpha)},
\end{align}
where we have set $h_0 = \mu + \nu + \gamma$, $h_1 = \alpha + \delta_A + \mu$, $h_2 = \delta_I + \mu$, $  h_3 = \gamma \alpha+(h_1 + \gamma ) h_2$, $h_4 = \dfrac{\gamma+\mu}{h_0}{\leq}1$. 
Moreover, we can compute
\begin{equation}\label{sum}
    \beta_A A^* + \beta_I I^* = \dfrac{\beta_A h_2 +\beta_I  \alpha}{\alpha}I^* = \dfrac{h_0 h_1 h_2 (\mathcal{R}_0-1)}{h_3}.
\end{equation}
To determine the stability of the endemic equilibrium $x^*$, we need to compute the Jacobian matrix of \eqref{sairs3_s} evaluated in $x^*$, that is
\begin{equation*}
    J_{|x^*} = \left(
    \begin{matrix}
  -  \dfrac{h_0 h_1 h_2 (\mathcal{R}_0-1)}{h_3} - h_0 & -\dfrac{\beta_A h_4}{\mathcal{R}_0}- \gamma & -\dfrac{\beta_I h_4}{\mathcal{R}_0}- \gamma \\
    \dfrac{h_0 h_1 h_2 (\mathcal{R}_0-1)}{h_3} & \dfrac{\beta_A h_4}{\mathcal{R}_0}-h_1 & \dfrac{\beta_I h_4}{\mathcal{R}_0} \\
    0 & \alpha & - h_2
    \end{matrix}
    \right),
\end{equation*}
where we have used (\ref{Sstar2}-\ref{sum}). With the same arguments as in \cite[Sec. 2.1]{robinson2013model}, we can conclude that $x^*$ is locally asymptotically stable if $\mathcal{R}_0 > 1$.
\end{proof}

\subsection{Global Stability}

\subsubsection{Global stability of the endemic equilibrium in the SAIR model.}\label{glob_sair}
In this section, we focus on the global asymptotic stability of the endemic equilibrium of the SAIR model, i.e., system (\ref{sairs3_s}) with $\gamma=0$, representing a disease which confers permanent immunity. 
Here, we answer directly to the open problem left in \cite{ansumali2020modelling}. Let us note that in our model we have in addition, with respect to the model proposed in \cite{ansumali2020modelling}, the possibility of vaccination.\\
The dynamic of an SAIR model of this type is described by the following system of equations:

\begin{align}\label{sair3}
    \frac{d S(t)}{dt} &= \mu -\bigg(\beta_AA(t) + \beta_II(t)\bigg)S(t) -(\mu + \nu) S(t) , \\ \nonumber
      \frac{d A(t)}{dt} &=\bigg(\beta_A A(t) + \beta_I I(t)\bigg)S(t) -(\alpha + \delta_A +\mu) A(t),  \\
     \frac{d I(t)}{dt} &= \alpha A(t) - (\delta_I + \mu)I(t), \nonumber 
    \end{align}
 The basic reproduction number is $${\mathcal{R}_0}  = \left ( \beta_A + \dfrac{\alpha \beta_I}{\delta_I + \mu}\right) \dfrac{ \mu}{(\alpha+\delta_A + \mu)(\nu + \mu)}.$$\\
% \ste{The existence of an unique endemic equilibrium in $\mathring \Gamma$ when ${\mathcal{R}_0}>1$ can be proved straightforwardly as in Lemma \ref{ex_ee}}.
 The endemic equilibrium $x^* = (S^*,A^*,I^*)$ satisfies the equation
 
 \begin{align}
 \mu&=\bigg(\beta_A A^* + \beta_I I^*\bigg)S^* +(\mu + \nu ) S^*, \label{cond1}\\
(\alpha + \delta_A +\mu) A^*&=  \bigg(\beta_A A^* + \beta_I(r)I^*\bigg)S^*, \label{cond2}\\
\alpha A^* &= (\delta_I + \mu)I^*. \label{cond3}
 \end{align}

\begin{theorem}
The endemic equilibrium $x^* = (S^*,A^*,I^*)$ of \eqref{sair3}  is globally asymptotically stable in $\mathring{\Gamma}$ if $\mathcal{R}_0>1$.
\end{theorem}

\begin{proof}
For ease of notation, we will omit the dependence on $t$. Let us consider $c_1, c_2 > 0$ and the function
$$V=c_1 V_1 + c_2 V_2 + V_3,$$
where  
\begin{equation*}
V_1= S^*\cdot g\left(\frac{S}{S^*}\right), \qquad V_2=A^*\cdot g\left(\frac{A}{A^*}\right), \qquad V_3=I^* \cdot g\left(\frac{I}{I^*}\right),    
\end{equation*}
and $g(x)=x-1-\ln x \geq g(1)=0$, for any $x >0$.
Let us introduce the notation 
\begin{equation*}
u= \frac{S}{S^*}, \qquad y=\frac{A}{A^*}, \qquad  z=\frac{I}{I^*}.    
\end{equation*}
Differentiating $V$ along the solutions of \eqref{sair3}, and using \eqref{cond1}, \eqref{cond2}, \eqref{cond3}, we have
 \begin{equation}\label{V1}
 \begin{split}
c_1 \frac{dV_1}{dt} =& c_1 \left(1-\frac{S^*}{S}\right)\bigg[\mu- (\beta_A A + \beta_I I)S -(\mu +\nu) S\bigg] \\ 
&=c_1 \left(1-\frac{S^*}{S}\right) \bigg[ -(\mu+\nu)(S-S^*) -\beta_A (AS -A^*S^*)-\beta_I (IS-I^*S^*)\bigg]\\
&= c_1 \left(1-\frac{1}{u}\right)\bigg[ -(\mu+\nu)S^*(u-1)- \beta_A A^*S^*(uy-1)- \beta_I I^*S^* (uz -1)\bigg],
\end{split}
 \end{equation}
 
 \begin{equation}\label{V2}
 \begin{split}
c_2 \frac{dV_2}{dt} =& c_2 \left(1-\frac{A^*}{A}\right)\bigg[ (\beta_A A + \beta_I I)S - (\alpha + \delta_A +\mu) A\bigg] \\
&=c_2\left(1-\frac{1}{y}\right)\bigg[ \beta_A A^*S^*uy+\beta_I I^*S^* uz -(\beta_A A^*+\beta_I I^*)S^* y\bigg]\\
&=c_2\left(1-\frac{1}{y}\right)\bigg[\beta_A A^*S^* (uy-y) +\beta_I I^*S^*(uz-y) \bigg],
\end{split}
\end{equation}

\begin{equation}\label{V3}
 \begin{split}
\frac{dV_3}{dt} =& \left(1-\frac{I^*}{I}\right)\bigg[ \alpha A-(\delta_I +\mu)I\bigg]= \left(1-\frac{I^*}{I}\right)\bigg(\alpha A-\frac{\alpha I A^*}{I^*}\bigg)\\
&= \alpha A^*\bigg( 1+\frac{A}{A^*} -\frac{I}{I^*} -\frac{AI^*}{A^*I}\bigg)\\
&\leq \alpha A^* \bigg( -\ln y +y -z +\ln z\bigg)\\
&=  \alpha A^*(g(y)-g(z)),
\end{split}
\end{equation}
where we have used the inequality $1-y/z \leq - \ln (y/z)$. Thus, from \eqref{V1},\eqref{V2}, and \eqref{V3},
\begin{equation}\label{dvdt}
\begin{split}
\frac{dV}{dt}=&-c_1 \left(1-\frac{1}{u}\right) (\mu+\nu)S^*(u-1) + c_1 \beta_A A^* S^* \bigg[   \left(1-\frac{1}{u}\right)(1-uy) +\frac{c_2}{c_1}\left(1-\frac{1}{y}\right)(uy-y)\bigg]\\
&+c_1  \beta_I I^* S^* \bigg[ \left(1-\frac{1}{u}\right)(1-uz) +\frac{c_2}{c_1}\left(1-\frac{1}{y}\right)(uz-y)\bigg]+\alpha A^*(g(y)-g(z)).
\end{split}
\end{equation}
Now, for the second and third term in \eqref{dvdt}, we have

\begin{equation}\label{g1}
\begin{split}
&\left(1-\frac{1}{u}\right)(1-uy) +\frac{c_2}{c_1}\left(1-\frac{1}{y}\right)(uy-y)\\
&=\left(1+\frac{c_2}{c_1}\right) -\frac{1}{u}-uy\left(1-\frac{c_2}{c_1}\right)+y\left(1-\frac{c_2}{c_1}\right)- \frac{c_2}{c_1}u\\
&=-g\left(\frac{1}{u}\right)-g\left(uy\right)\left(1-\frac{c_2}{c_1}\right)+ \left(g(y)\left(1-\frac{c_2}{c_1}\right) -g(u)\right),
\end{split}
\end{equation}
and
\begin{equation}\label{g2}
\begin{split}
&\left(1-\frac{1}{u}\right)(1-uz) +\frac{c_2}{c_1}\left(1-\frac{1}{y}\right)(uz-y)\\
&=\left(1+\frac{c_2}{c_1}\right) -\frac{1}{u}+z -uz\left(1-\frac{c_2}{c_1}\right) - \frac{c_2}{c_1} y-\frac{c_2}{c_1} \frac{uz}{y}\\
=& -g\left(\frac{1}{u}\right) - \frac{c_2}{c_1} g\left(\frac{uz}{y}\right)+\left(g(z)- \frac{c_2}{c_1} g(y)\right)-uz\left(1-\frac{c_2}{c_1}\right).
\end{split}
\end{equation}
Thus, substituting \eqref{g1} and \eqref{g2} in \eqref{dvdt}, we obtain
\begin{equation*}
\begin{split}
\frac{dV}{dt}=& -c_1 \left(1-\frac{1}{u}\right) (\mu+\nu)S^*(u-1) \\
&-c_1 \beta_A A^* S^* \bigg[g\left(\frac{1}{u}\right) +  g(uy) \left(1-\frac{c_2}{c_1}\right)\bigg]
+c_1 \beta_A A^* S^* \bigg[ g(y)\left(1-\frac{c_2}{c_1}\right) -g(u)\bigg]\\
&-c_1\beta_I I^* S^* \bigg[g\left(\frac{1}{u}\right) +\frac{c_2}{c_1} g\left(\frac{uz}{y}\right)\bigg] +c_1\beta_I I^* S^*\bigg[ g(z)-\frac{c_2}{c_1} g(y) -uz\left(1-\frac{c_2}{c_1}\right)\bigg]\\
& +\alpha A^*(g(y)-g(z)).
\end{split}
\end{equation*}
Now, by taking $c_1=c_2=\frac{\alpha A^*}{\beta_I I^* S^*}$, we have

\begin{align*}
   \frac{dV}{dt}=&-c_1 \frac{(u-1)^2}{u}(\mu+\nu)S^* -c_1 \beta_A A^* S^*\bigg( g\left(\frac{1}{u}\right)+g(u) \bigg)\\
   &-c_1\beta_I I^* S^* \bigg(g\left(\frac{1}{u}\right) + g\left(\frac{uz}{y}\right) \bigg). 
\end{align*}
Hence, $\frac{dV}{dt} \leq 0$. Moreover, the set where $\frac{dV}{dt}=0$ is
    $Z=\{ (S,A,I): S=S^*, I=\frac{AI^*}{A^*}\},$ and the only compact invariant subset of $Z$ is the singleton $\{x^*\}$. The claim follows by LaSalle's Invariance Principle \cite{la1976stability}.
\end{proof}

\subsubsection{Global stability of the SAIRS model when $\beta_A = \beta_I := \beta$ and $\delta_A = \delta_I := \delta$}\label{glob_sairs_equal}

In this case, from \eqref{R0}, the expression of the basic reproduction number becomes
\begin{equation*}
 \mathcal{R}_0  = \frac{\beta(\gamma+\mu)}{(\delta + \mu)(\nu +\gamma+\mu)}.
\end{equation*}

\begin{theorem}
Let us assume that $\beta_A = \beta_I =: \beta$ and $\delta_A = \delta_I =: \delta$. The endemic equilibrium $x^* = (S^*,A^*,I^*)$ is globally asymptotically stable in $\mathring{\Gamma}$ for system \eqref{sairs3_s} if $ \mathcal{R}_0>1$.
\end{theorem}
\begin{proof}

Let us define $M(t) := A(t) + I(t)$, for all $t\geq 0$. Then, we can rewrite \eqref{sairs3_s} as
\begin{align*}
    \frac{d S(t)}{dt} &= \mu - \beta M(t) S(t) -(\mu + \nu +\gamma) S(t) + \gamma(1-M(t) ), \\ \nonumber
     \frac{dM(t)}{dt} &= \beta M(t) S(t) -  (\delta + \mu)M(t).
    \end{align*}

%If $\mathcal{R}_0 > 1$, from Lemma $\ref{ex_ee}$ we know that the endemic equilibrium $x^* = (S^*,A^*,I^*)$ exists. Then,
At the equilibrium it holds that
\begin{align}
    \mu + \gamma &= \beta M^*S^* + (\mu + \nu + \gamma) S^* + \gamma M^*, \label{Icond}\\
   \delta + \mu& =  \beta S^*, \label{IIcond}
\end{align}
where $M^* = A^* + I^*$.
In the following, for ease of notation, we will omit the dependence on $t$. Consider the following positively definite function
\begin{equation*}\label{Lyap1}
    V = \frac{1}{2}(S - S^*)^2 + w\left(M- M^* - M^*\ln\left(\dfrac{M}{M^*}\right)\right) ,
\end{equation*}
where $w$ is a non negative constant.

Differentiating along \eqref{sairs3_s} and using the equilibrium  conditions (\ref{Icond}-\ref{IIcond}) we obtain
\begin{align*}
    \frac{d V}{dt}= &(S-S^*) \left( \beta (M^*S^*-MS) -(\mu + \nu + \gamma )(S-S^*) \right.+ \\
    & +\left. \gamma (M^*-M) \right)  
    + w \left( 1 - \frac{M^*}{M}\right) \beta M (S-S^*) \\
    = & \;  \beta (S- S^*) (M^*S^* - MS^* + M S^* - M S ) - (\mu + \nu + \gamma) (S- S^*)^2 + \\
    & + \gamma(M^* - M ) (S - S^*) + w \beta (M - M^*) (S - S^*) \\
    = & \; \beta S^* ( S - S^*) (M^*- M ) - ( \beta M+ \mu + \nu + \gamma) (S- S^*)^2 + \\
    & + \gamma (M^*-M) (S - S^*) + w \beta(M - M^*) (S- S^*) \\
   \leq & \;  \left( \beta S^*  + \gamma - w \beta \right) (S - S^*) (M - M^*).
\end{align*}
Choosing $w := \frac{\beta S^* + \gamma}{\beta} > 0$, it follows that $\frac{dV}{dt} \leq 0$. {The claim follows from the same argument used in \cite[Thm 7]{ansumali2020modelling}.} \end{proof}

 \subsubsection{Global stability of the SAIRS model with $\beta_A \neq \beta_I$ and $\delta_A \neq \delta_I$: a geometric approach}\label{glob_SAIRS_noequal}
 
 We use a geometric approach for the global stability of equilibria of nonlinear autonomous differential equations proposed in \cite{lu2017geometric}, that is a generalization of the approach developed by Li and Muldowney \cite{li1996geometric,li2000dynamics}. First, we briefly recall the salient concepts.
 
Consider the following autonomous system 
\begin{equation}\label{syst_geo}
 x'= f(x), \qquad x \in D \subset \mathbb R^n,     
\end{equation}
 where $f(x): D \to \mathbb R^n $ is a continuous differentiable function in $D$. Let $x(t,x(0))$ be the solution of system \eqref{syst_geo} with the initial value $x(0,x(0))=x(0)$.
 We assume that system \eqref{syst_geo} has an $n-m$ dimensional invariant manifold $\Omega$ defined by
 \begin{equation}\label{omega_inv}
 \Omega=\{ x \in \mathbb R^n| g(x)=0\},
 \end{equation}
 where $g(x)$ is an $\mathbb R^m$-valued twice continuously differentiable function with dim$(\frac{\partial g}{\partial x})=m$ when $g(x)=0$.
 In \cite{li2000dynamics}, Li and Muldowney proved that if $\Omega$ is invariant with respect
to system \eqref{syst_geo}, then there exists a continuous $m \times m$ dimensional
matrix-valued function $N(x)$, such that 
$$g_f(x)=\frac{\partial g}{\partial x}\cdot f(x)=N(x)\cdot g(x),$$
 where $g_f(x)$ is the directional derivative of $g(x)$ in the direction of the vector field $f$.
 Moreover, let us define the real valued function $\sigma(x)$ on $\Omega$, by 
 $$\sigma(x)=tr(N(x)),$$
 and make the following assumptions:
 
 \begin{itemize}
     \item[(H1)] $\Omega$ is simply connected; 
     \item[(H2)] There is a compact absorbing set $K \subset D \subset \Omega$;
     \item[(H3)] $x^*$ is the unique equilibrium of system \eqref{syst_geo} in $D \subset \Omega$ which satisfies $f(x^*)=0$.
 \end{itemize}
 
 Now, consider the following linear differential equation, associated to system \eqref{syst_geo}
 \begin{equation}\label{z_syst}
     z'(t)=\left[ P_f P^{-1} + P J^{[m+2]}P^{-1}-\sigma I\right]z(t)=: B(x(t,x(0))z(t),
 \end{equation}
 where $x \mapsto P(x)$ is a $C^1$ nonsingular ${n \choose m+2} \times {n \choose m+2}$ matrix-valued function in $\Omega$ such that $||P^{-1}(x)||$ is uniformly bounded for $x \in K$ and $P_f$ is the directional derivative of $P$ in the direction of the vector field $f$, and $J^{[m+2]}$ is the $m+2$ additive compound matrix of the Jacobian matrix of \eqref{syst_geo}.
 Assume that the following additional condition holds:\\
 
 \noindent
 (H4) for the coefficient matrix $B(x(t,x(0))$, there exists a matrix $C(t)$, a large enough $T_1>0$ and some positive numbers $\alpha_1,\alpha_2, \ldots, \alpha_n$ such that for all $t \geq T_1$ and all $x(0) \in K$ it holds
 \begin{equation*}
b_{ii}(t)+\sum_{i \neq j} \frac{\alpha_j}{\alpha_i}|b_{ij}(t)| \leq c_{ii}(t) + \sum_{i \neq j}  \frac{\alpha_j}{\alpha_i} |c_{ij}(t)|,
 \end{equation*}
and
\begin{equation*}
\lim_{t \to \infty} \frac{1}{t} \int_0^t c_{ii}(s) + \sum_{i \neq j}  \frac{\alpha_j}{\alpha_i} |c_{ij}(s)|\; ds= h_i <0, 
\end{equation*}
where $b_{ij}(t)$ and $c_{ij}(t)$ represent entries of matrices $B(x(t,x(0))$ and $C(t)$, respectively.
Basically, condition (H4) is a Bendixson criterion for ruling out non-constant periodic solutions of system \eqref{syst_geo} with invariant manifold $\Omega$. 
 From this, by a similar argument as in Ballyk et al. \cite{ballyk2005global}, based on \cite[Thm 6.1]{li2000dynamics}, the following theorem can be deduced (see \cite[Thm 2.6]{lu2017geometric}).
 
 \begin{theorem}\label{thm26}
 Under the assumptions (H1)-(H4), the unique endemic equilibrium $x^*$ of \eqref{syst_geo} is globally asymptotically stable in $D \subset \Omega$.
 \end{theorem}
 
 For our system \eqref{sairs_s}, we have that the invariant manifold \eqref{omega_inv} is the set $\bar \Gamma$ in \eqref{gamma_inv}, %g(x)=S+A+I+R-1, %$\frac{\partial g}{\partial x}=(1,1,1,1)$
 so $n=4$, $m=1$, and $N(x)=-\mu$. It is easy to see that (H1) holds, and that for $\mathcal R_0>1$, by Theorem \ref{uniform_per} and Lemma \ref{ex_ee}, (H2)-(H3) follows.
 
 \begin{theorem}\label{thmcond}
 Assume that $\mathcal{R}_0 > 1$ and $\beta_A < \delta_I$. Then, the endemic equilibrium $x^*$ is globally asymptotically stable in $\mathring{\bar \Gamma}$ for system \eqref{sairs_s}.
 \end{theorem}

\begin{proof}
Let us recall that from \eqref{pers}, there exists $T>0$ such that for $t > T$,

\begin{equation}\label{pers2}
    \eps \leq S(t),A(t),I(t),R(t) \leq 1-\eps.
\end{equation}
The Jacobian matrix of \eqref{sairs_s} may be written as 
\begin{equation*}
    J = -\mu I_{4 \times 4} + \Phi,
\end{equation*}
where $I_{4 \times 4}$ is the $4\times 4$
identity matrix and 
\begin{equation*}
    \Phi = \left(
\begin{matrix}
    -(\beta_A A +\beta_I I +\nu) & -\beta_A S & -\beta_I S & \gamma\\
    \beta_A A + \beta_I I & \beta_A S -(\delta_A + \alpha) & \beta_I S & 0 \\
    0 & \alpha & -\delta_I & 0 \\
    \nu & \delta_A & \delta_I & -\gamma
\end{matrix}
    \right).
\end{equation*}
From the definition of the third additive compound matrix (see, e.g., \cite[Appendix]{li1999global}), we have $$J^{[3]}=\Phi^{[3]}-3 \mu I_{4 \times 4},$$
with
$$\Phi^{[3]}= \left(\phi_1^{[3]},\phi_2^{[3]},\phi_3^{[3]},\phi_4^{[3]}\right)^T,$$
where

$$\phi_1^{[3]}=\left(-(\beta_A A +\beta_I I +\nu) + \beta_A S-(\delta_A +\alpha) -\delta_I, \; 0,\; 0,\; \gamma\right)^T,$$
$$\phi_2^{[3]}= \left(\delta_I,\; -(\beta_A A +\beta_I I +\nu) + \beta_A S-(\delta_A +\alpha) -\gamma,\; \beta_I S,\; \beta_I S\right)^T,$$
$$\phi_3^{[3]}= \left(-\delta_A, \; \alpha, \; -(\beta_A A +\beta_I I +\nu) -\delta_I -\gamma, \; -\beta_A S\right)^T,$$
$$\phi_4^{[3]}= \left(\nu, \; 0, \; \beta_A A +\beta_I I, \; \beta_A S -(\delta_A +\alpha+ \delta_I+ \gamma)\right)^T.$$
\begin{comment}
\begin{equation*}
\Phi^{[3]} = \left(
\begin{matrix}
    -(\beta_A A +\beta_I I +\nu) + \beta_A S-(\delta_A +\alpha) -\delta_I & 0 & 0 & \gamma\\[5pt]
    \delta_I & -(\beta_A A +\beta_I I +\nu) + \beta_A S-(\delta_A +\alpha) -\gamma & \beta_I S & \beta_I S \\[5pt]
    -\delta_A & \alpha & -(\beta_A A +\beta_I I +\nu) -\delta_i -\gamma & -\beta_A S \\[5pt]
    \nu & 0 & \beta_A A +\beta_I I & \beta_A S -(\delta_A +\alpha- \delta_I- \gamma)
\end{matrix}
    \right).
\end{equation*}
\end{comment}
Let $P(x)$ be such that 
$$P(x)=\diag(R,cI,A,S),$$
where $c$ is a constant such that $\frac{\delta_I+\mu}{\beta_I \eps+\nu +\delta_I+\mu}<c<1$, then from \eqref{z_syst} by direct computation we have 
\begin{equation*}
B(t)=P_f P^{-1} + P J^{[3]}P^{-1}+\mu I_{4 \times 4} = 
\diag\left(\frac{R'}{R},\frac{I'}{I},\frac{A'}{A},\frac{S'}{S}\right)+ P\Phi^{[3]}P^{-1}-2\mu I_{4 \times 4}, 
\end{equation*}
where
$$P\Phi^{[3]}P^{-1}=\left(\zeta_1^{[3]},\zeta_2^{[3]},\zeta_3^{[3]},\zeta_4^{[3]}\right)^T,$$
and
$$\zeta_1^{[3]}=\left(-(\beta_A A +\beta_I I +\nu) + \beta_A S-(\delta_A +\alpha) -\delta_I, \; 0,\; 0,\; \gamma \frac{R}{S}\right)^T,$$
$$\zeta_2^{[3]}= \left(c \frac{\delta_I I}{R},\; -(\beta_A A +\beta_I I +\nu) + \beta_A S-(\delta_A +\alpha) -\gamma,\; c \frac{\beta_I I S}{A},\; c \beta_I I\right)^T,$$
$$\zeta_3^{[3]}= \left(-\frac{\delta_A A}{R}, \; \frac{\alpha A}{c I}, \; -(\beta_A A +\beta_I I +\nu) -\delta_I -\gamma, \; -\beta_A A\right)^T,$$
$$\zeta_4^{[3]}= \left(\frac{\nu S}{R}, \; 0, \; (\beta_A A +\beta_I I)\frac{S}{A}, \; \beta_A S -(\delta_A +\alpha+ \delta_I+ \gamma)\right)^T.$$
From the system of equations \eqref{sairs_s}, we obtain
\begin{equation}\label{SA}
\frac{\gamma R}{S}= \mu\left(1- \frac{1}{S}\right)+(\beta_A A + \beta_I I) + \nu +\frac{S'}{S} , \qquad \frac{\beta_I I S}{A}= \alpha+\delta_A + \mu - \beta_A S + \frac{A'}{A},   
\end{equation}
\begin{equation}\label{IR}
    \frac{\alpha A}{I}= \delta_I +\mu + \frac{I'}{I}, \qquad \frac{\delta_I I}{R}= \gamma + \mu -\frac{\delta_I I}{R} -\frac{\nu S}{R}+\frac{R'}{R}.
\end{equation}
Consequently, by using \eqref{pers2} and \eqref{SA}-\eqref{IR}, we have  
\begin{equation*}
\begin{split}
    h_1(t)&= b_{11}(t)+ \sum_{j \neq 1}|b_{1j}(t)|\\
    &=-(\beta_A A +\beta_I I+ \nu) + \beta_A S- (\delta_A +\alpha) -\delta_I -2 \mu + \frac{R'}{R} +\frac{\gamma R}{S}\\
    &= \beta_A S-\delta_A -\alpha -\delta_I -\frac{\mu}{S}+\frac{R'}{R}+\frac{S'}{S} \\
    &\leq \beta_A -\delta_A -\alpha -\delta_I+\frac{R'}{R}+\frac{S'}{S}=: \bar h_1(t), \\ \\
    h_2(t)&= b_{22}(t)+ \sum_{j\neq2}|b_{2j}(t)|\\
    &=-(\beta_A A +\beta_I I+ \nu)+ \beta_A S- (\delta_A +\alpha) -\gamma -2 \mu + \frac{I'}{I} + c\frac{\delta_I I}{R} + c\frac{\beta_I S I}{A} + c \beta_I I \\
    & \leq -\eps \beta_A  - \nu -\gamma -\mu +c(\gamma + \mu) +c \frac{I'}{I}+ c \frac{R'}{R} + \frac{A'}{A}=: \bar h_2(t),\\ \\
    h_3(t)&= b_{33}(t)+ \sum_{j\neq 3}|b_{3j}(t)|\\
    &= -(\beta_A A +\beta_I I+ \nu) -\delta_I-\gamma -2 \mu +\frac{A'}{A} +\frac{\delta_A A}{R}+ \frac{\alpha A}{c I}+ \beta_A A\\
    &\leq - \eps \beta_I - \nu - \delta_I -\mu + \frac{\delta_I + \mu}{c} + \frac{A'}{A}+ \frac{R'}{R}+ \frac{I'}{cI}
    =: \bar h_3(t),\\ \\
    h_4(t)&= b_{44}(t)+ \sum_{j\neq 4}|b_{4j}(t)|\\
    &=\beta_A S -(\delta_A +\alpha) -\delta_I -\gamma -2 \mu + \frac{S'}{S} + \frac{\nu S}{R} + \beta_A S + \frac{\beta_I S I}{A}\\
    & \leq -\delta_I + \beta_A + \frac{S'}{S}+ \frac{R'}{R} + \frac{A'}{A} =: \bar h_4(t).
    \end{split}
\end{equation*}
Then, we can take the matrix $C$ in condition (H4) as
$$C(t)=\diag\left(\bar h_1(t), \bar h_2(t), \bar h_3(t),\bar h_4(t)\right),$$
based on \eqref{pers2} and by the assumption $\beta_A < \delta_I$, we can assert that
$$\lim_{t \to \infty} \frac{1}{t} \int_0^t \bar h_i(s) ds = \bar h_i <0, \qquad i=1,\ldots,4,$$
where 
\begin{equation*}
\bar h_1=\beta_A -\delta_A -\alpha -\delta_I, \quad \bar h_2= -\eps \beta_A  - \nu -\gamma -\mu +c(\gamma + \mu) , \quad \bar h_3= - \eps \beta_I - \nu - \delta_I -\mu+ \frac{\delta_I + \mu}{c}, \quad \bar h_4= -\delta_I + \beta_A.    \end{equation*}
Indeed, if $\beta_A < \delta_I$ holds, both $\bar h_2$ and $\bar h_3$ are less than zero; moreover, $\bar h_1$ and $\bar h_2$ are less than zero by the choice of $c$.
The claim then follows from Theorem \ref{thm26}.
\end{proof} 

%\sara{Notice that in \cite[Theorem 3.4]{lu2017geometric} the global asymptotic stability of the endemic equilibrium for the SEIRS model has been proved without any condition on the parameters. Using the same technique, 

{We proved the global asymptotic stability of the endemic equilibrium for the SAIRS model with a condition on the parameters, that is $\beta_A < \delta_I$.}
{However, supported also by numerical simulations in Sec.~\ref{num_analysis_sec}, we are led to think that this assumption could be relaxed. Thus, we state the following conjecture.}

\begin{conjecture}\label{conj}
%Assume $\mathcal{R}_0>1$. Then, the assumption $\beta_A < \delta_I$ in Theorem \ref{thmcond} is sufficient but not necessary for the endemic equilibrium to be globally asymptotically stable for system \eqref{sairs_s}.
The endemic equilibrium $x^*$ is globally asymptotically stable in $\mathring{\bar \Gamma}$ for system \eqref{sairs_s} if $R_0>1$.
\end{conjecture}

\subsection{SAIRS without vaccination ($\nu =0$).}\label{no_vax}
Let us note that in the SAIRS-type models proposed so far, we have obtained results for the global stability of the DFE equilibrium when $\mathcal{R}_0 <1$ and for the global stability of the endemic equilibrium when $\mathcal{R}_0>1$ (plus eventually a further conditions), but we are not able to study the stability of our system in the case $\mathcal{R}_0=1$. However, if we consider the SAIRS model without vaccination, i.e. the model \eqref{sairs3_s} with $\nu =0$, we are able to study also the case $\mathcal{R}_0=1$. %gives a sharp threshold \cite{shuai2013global}. Precisely, i
From (\ref{R0}), in the case $\nu =0$, we have 
\begin{equation}\label{R0novax}
    \mathcal{R}_0 = \left ( \beta_A + \dfrac{\alpha \beta_I}{\delta_I + \mu}\right) \dfrac{1}{(\alpha+\delta_A + \mu)},
\end{equation}
the DFE is $x_0=(1,0,0)$,
and we obtain the following result.

\begin{theorem}
The disease-free equilibrium $x_0$ is global asymptotically stable if $\mathcal{R}_0\leq 1$.
\end{theorem}

\begin{proof}
We follow the idea in \cite[Prop. 3.1]{guo2006global}.
Let
\begin{equation*}
C = \left( \begin{matrix} 
\alpha + \delta_A + \mu & 0\\ 
- \alpha & \delta_I + \mu
\end{matrix} \right), 
\end{equation*}
and 
$$Y=(A,I)^T.$$
Thus, we have
$$\frac{ dY}{dt} = \left(C  (M(S)-I_{2 \times 2})\right)Y, $$
where 
\begin{equation*}
M(S) = \left( \begin{matrix} 
\frac{\beta_A S}{\alpha + \delta_A + \mu}  & \frac{\beta_I S}{\alpha + \delta_A + \mu} \\ \\
\frac{\alpha \beta_A S}{(\delta_I+\mu)(\alpha + \delta_A + \mu)}  & \frac{\alpha \beta_I S}{(\delta_I+\mu)(\alpha + \delta_A + \mu)}
\end{matrix} \right).
\end{equation*}

Since, in this case, $S_0=1$, we have that $ 0 \leq S\leq S_0$, and $0\leq M(S) \leq M(S_0)$, meaning that each element of $M(S)$ is less than or equal to the corresponding element of $M(S_0)$.\\ %$M(S)$ is irreducible for all $S>0$, thus $\rho(M(S)) < \rho(M(S_0))$, provided $S \neq S_0$ (see, e.g., \cite{varga1962matrix}).\\
At this point, let us consider the positive-definite function
$$V(Y)=w\;C^{-1}Y,$$
where $w$ is the left-eigenvector of $M(S_0)$ corresponding to $\rho(S_0)$; since $M(S_0)$ is a positive matrix, by Perron's theorem, $w>0$. It is easy to see that $\rho(M(S_0))=\mathcal{R}_0$ in \eqref{R0novax}, thus if $\mathcal{R}_0\leq 1$, we have
 \begin{align*}
     \frac{dV}{dt} &=w \;C^{-1} \frac{dY}{dt} = w \left( M(S)-I_{2 \times 2} \right)Y\\
     &\leq w \left( M(S_0)-I_{2 \times 2} \right)Y 
     = (\rho(M(S_0))-1)w Y \leq 0.
 \end{align*}
If $\mathcal{R}_0<1$, then  $\frac{dV}{dt}=0 \iff Y=0$. If $\mathcal{R}_0=1$, then %$w(M(S)-Id)Y=0$ implies 
\begin{equation}\label{wM}
wM(S)Y=wY.    
\end{equation}
Now, if $S \neq S_0$, $w M(S) < wM(S_0)=\rho(M(S_0))w=w$: Thus, \eqref{wM} holds if and only if $Y=0$. If $S=S_0$, $wM(S)=wM(S_0)=w$, and $\frac{dV}{dt}=0$ if $S=S_0$ and $Y=0$. It can be seen that
the maximal compact invariant set where $\frac{dV}{dt}=0$ is the singleton $\{x_0\}$. Thus, by the LaSalle invariance principle the DFE $x_0$ is globally asymptotically stable if $\mathcal{R}_0 \leq 1$.
\end{proof}

\section{Numerical analysis}\label{num_analysis_sec}

In this Section, we provide numerous realizations of system (\ref{sairs_s}). In particular, to back the claim we made in Conjecture \ref{conj}, in all the figures we chose $\beta_A>\delta_I$, with the exception of Figure \ref{fig:tnucond}, still obtaining numerical convergence towards the endemic equilibrium when $\mathcal{R}_0>1$.\\
Considering all the other parameters to be fixed, $\mathcal{R}_0$ becomes a linear function of $\beta_A$ and $\beta_I$; in particular, the line $\mathcal{R}_0(\beta_A,\beta_I)=1$ is clearly visible in all the subfigures of Figure \ref{fig:surf}, in which we visualize the equilibrium values of $S,A,I,R$ as functions of $\beta_A$ and $\beta_I$.
When $R_0<1$, the values of $\beta_A$ and $\beta_I$ do not influence the value of the equilibrium point \eqref{eq:DFE}, and the value of the fraction of individuals in each compartment remains constant. For values of $R_0>1$, we can see the the influence of the infection parameters on each components of the endemic equilibrium (see (\ref{Astar}), (\ref{Sstar}), (\ref{Istar})).

\begin{figure}[H]
    \begin{subfigure}{.49\textwidth}
        \centering
        \includegraphics[width=0.9\textwidth]{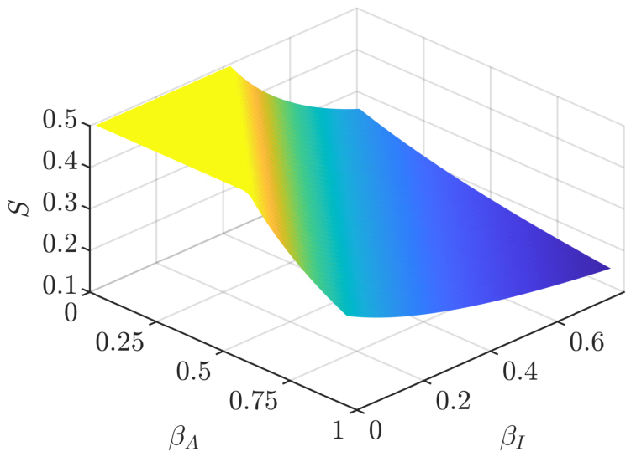}
        \caption{}\label{fig:surfS}
    \end{subfigure}\hfill
    \begin{subfigure}{0.49\textwidth}
        \centering
        \includegraphics[width=0.9\textwidth]{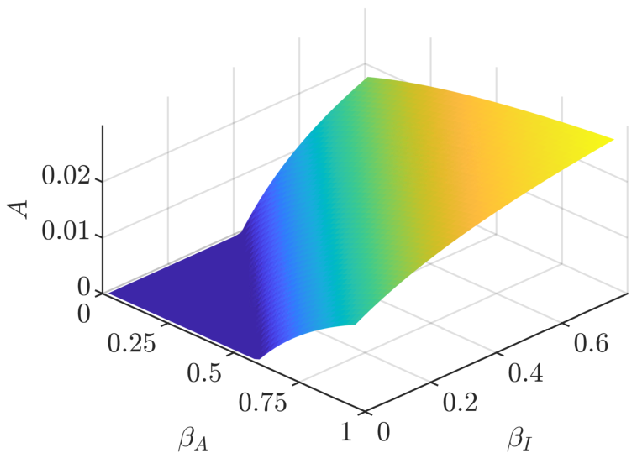}
        \caption{}\label{fig:surfA}
    \end{subfigure}
    \begin{subfigure}{0.49\textwidth}
        \centering
        \includegraphics[width=0.9\textwidth]{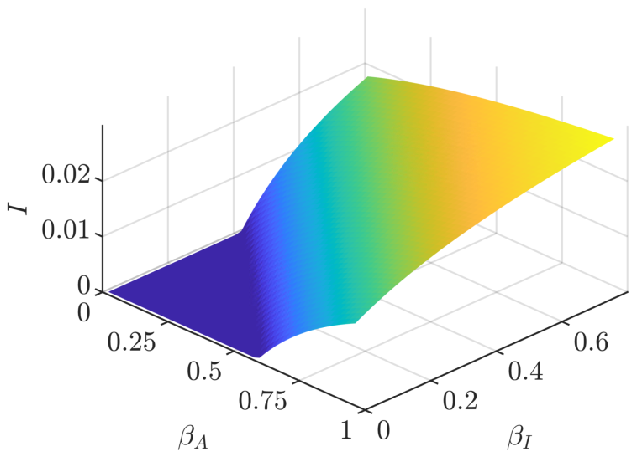}
        \caption{}\label{fig:surfI}
    \end{subfigure}\hfill
    \begin{subfigure}{0.49\textwidth}
        \centering
        \includegraphics[width=0.9\textwidth]{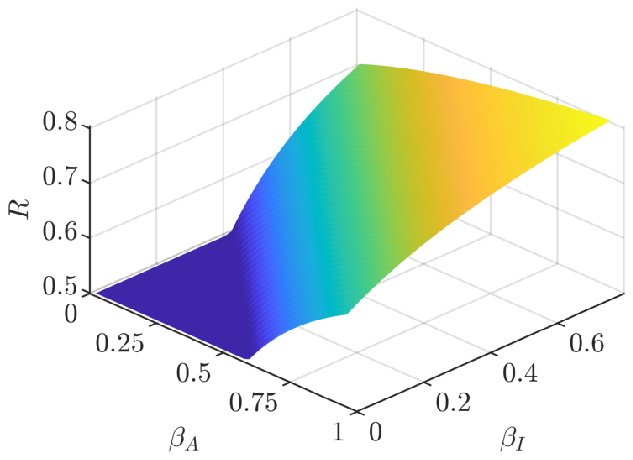}
        \caption{}\label{fig:surfR}
    \end{subfigure}
    \caption{Asymptotic values of $S$, $A$, $I$, and $R$ as a function of $\beta_A$ and $\beta_I$. Values of the parameters: $\mu=1/(70\cdot365)$, meaning an average lifespan of 70 years; $\beta_A \in [0.01,0.8]$ $\beta_I \in [0.01,0.95]$, $\nu=0.01$, $\gamma=1/100$, meaning the immunity lasts on average 100 days; $\alpha=0.15$, $\delta_A=0.1$, $\delta_I=0.15$.}
\label{fig:surf}
\end{figure}

Figures \ref{fig:stgamma}, \ref{fig:atgamma}, \ref{fig:itgamma} and \ref{fig:rtgamma} confirm our analytical results on the asymptotic values of the fraction of individuals in each compartment. In particular, the endemic equilibrium value of $S$ (\ref{Sstar}) does not depend on $\gamma$, the loss of immunity rate, as shown by the time series corresponding to $\gamma=0.01$, $0.02$ and $0.05$, whereas the disease free equilibrium value of $S$ (\ref{eq:DFE}), corresponding to the $\gamma=0.001$ plot, does. %Decreasing 
Increasing the value of $\gamma$, which corresponds to decreasing the average duration $1/\gamma$ of the immunity time-window, results in %smaller 
bigger asymptotic values for the asymptomatic and symptomatic infected population $A$ and $I$ and in a %bigger
smaller asymptotic value for the recovered population $R$. This trend is quite intuitive: indeed, by keeping the others parameters fixed, if the average immune period
decreases (i.e., $\gamma$ increases), a removed individual quickly return to the susceptible state, hence the behavior of the SAIRS model approaches that of a SAIS model.

\begin{figure}[H]
    \centering
    \begin{subfigure}{0.49\textwidth}
        \centering
        \includegraphics[width=0.9\textwidth]{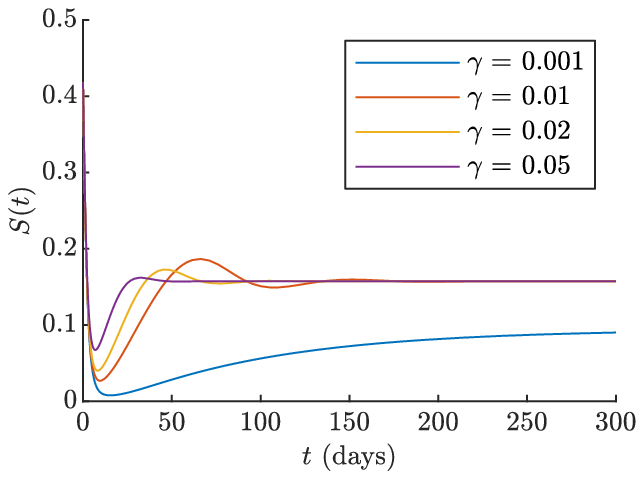}
        \caption{}\label{fig:stgamma}
    \end{subfigure}\hfill
    \begin{subfigure}{0.49\textwidth}
        \centering
        \includegraphics[width=0.9\textwidth]{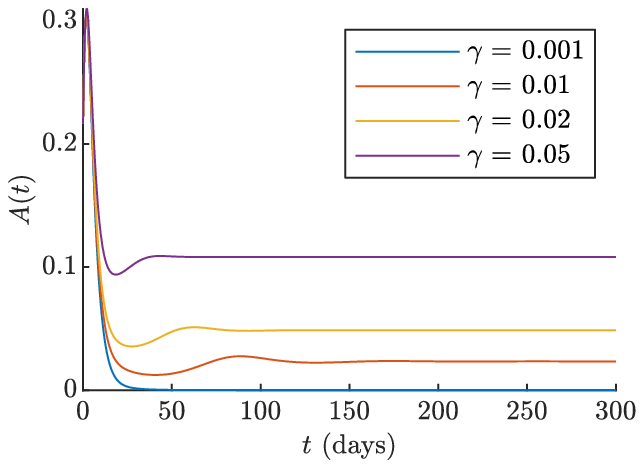}
        \caption{}\label{fig:atgamma}
    \end{subfigure}
    \begin{subfigure}{0.49\textwidth}
        \centering
        \includegraphics[width=0.9\textwidth]{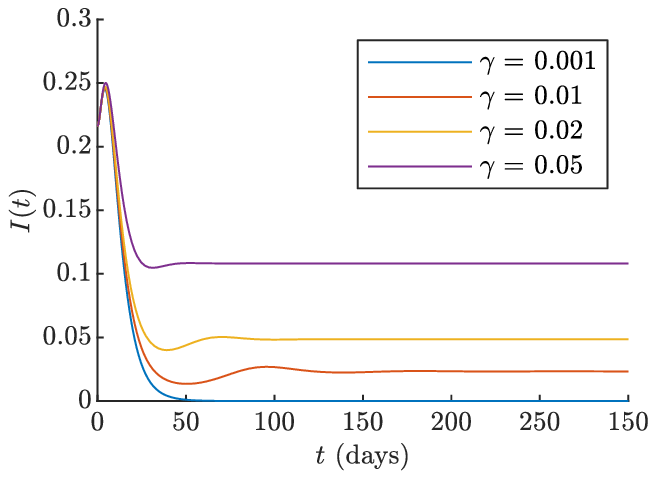}
        \caption{}\label{fig:itgamma}
    \end{subfigure}\hfill
    \begin{subfigure}{0.49\textwidth}
        \centering
        \includegraphics[width=0.9\textwidth]{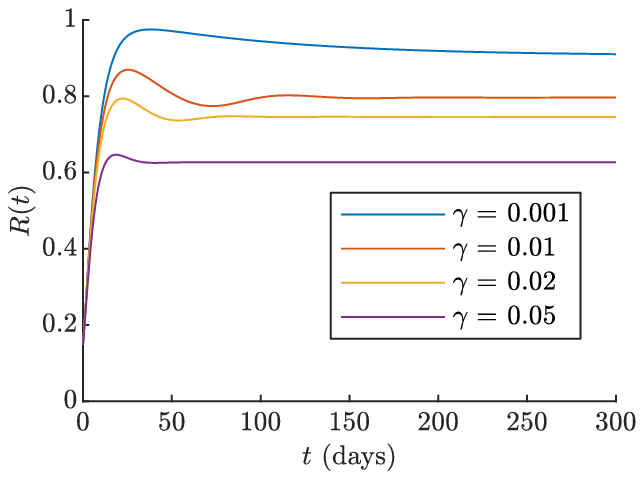}
        \caption{}\label{fig:rtgamma}
    \end{subfigure}
    \caption{Behavior of system (\ref{sairs_s}) as $\gamma$, the rate of loss of immunity, varies. Values of the parameters: $\mu=1/(70\cdot365)$, meaning an average lifespan of 70 years; $\beta_A =0.8$ $\beta_I =0.95$, $\nu=0.01$, $\gamma$ varying as shown; $\alpha=0.15$, $\delta_A=0.125$, $\delta_I=0.15$.}
    \label{fig:tgamma}
\end{figure}

Next, we explore the effect of changing $\alpha$, {the rate of symptoms onset}, in three scenarios: equally infectious asymptomatic and symptomatic individuals ($\beta_A=\beta_I$), in Figure \ref{fig:talpha99}; asymptomatic individuals more infectious than symptomatic individuals ($\beta_A>\beta_I$) (this case can be of interest if we consider that asymptomatic individuals can, in principle, move and spread the infection more than symptomatic ones) in Figure \ref{fig:talpha95}; and vice-versa ($\beta_A<\beta_I$), in Figure \ref{fig:talpha59}. If $\mathcal{R}_0>1$, $A^*$ and $I^*$ are related by $A^*=\frac{\delta_I+\mu}{\alpha}I^*$ (\ref{Astar}). This means that, regardless of the values of $\beta_A$ and $\beta_I$, $A^*>I^*$ if and only if $\frac{\delta_I+\mu}{\alpha}>1$. This is evident in Figures \ref{fig:at99}, \ref{fig:at95} and \ref{fig:at59}, where the smallest value of that ratio, corresponding to $\alpha=0.9$, is smaller than $1$, results in $I^*>A^*$; the biggest value of that ratio, and the only one significantly bigger than $1$ is attained for $\alpha=0.01$, and results in $I^*<A^*$.
Increasing $\alpha$ leads to a smaller asymptotic value for $A$, and a bigger asymptotic value for $I$.
Effectively, by keeping fixed the other parameters and increasing $\alpha$ leads to a decreasing of the average time-period before developing symptoms, thus the behavior of the SAIRS model approaches that of the SIRS one, as $\alpha$ increases.

Finally, in Figure \ref{fig:tnucond}, we compare the effect of varying $\nu$, the vaccination rate, on the {epidemic} dynamics. In particular, the parameter values chosen satisfy the assumption of Theorem \ref{thmcond}, i.e. $\mathcal{R}_0>1$ and simultaneously $\beta_A<\delta_I$. We observe that the asymptotic values of $A$ and $I$ are decreasing in $\nu$, whereas the endemic equilibrium value of $S$ is independent from this parameter, as we expect from (\ref{Sstar}), and the endemic equilibrium value of $R$ is increasing in $\nu$.

\begin{figure}[H]
    \centering
    \begin{subfigure}{0.49\textwidth}
        \centering
        \includegraphics[width=0.9\textwidth]{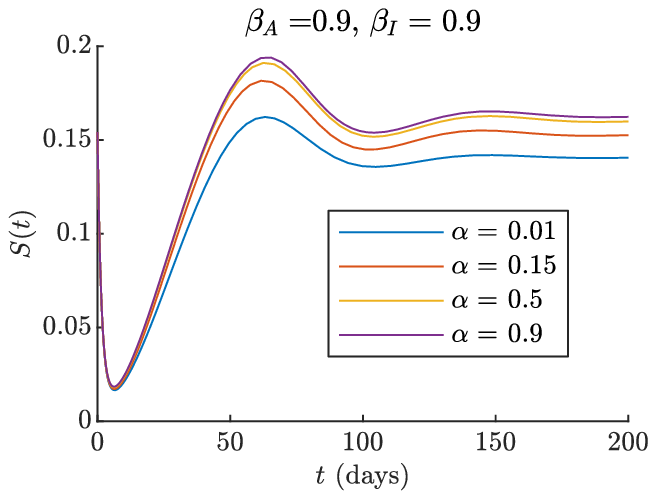}
        \caption{}\label{fig:st99}
    \end{subfigure}\hfill
    \begin{subfigure}{0.49\textwidth}
        \centering
        \includegraphics[width=0.9\textwidth]{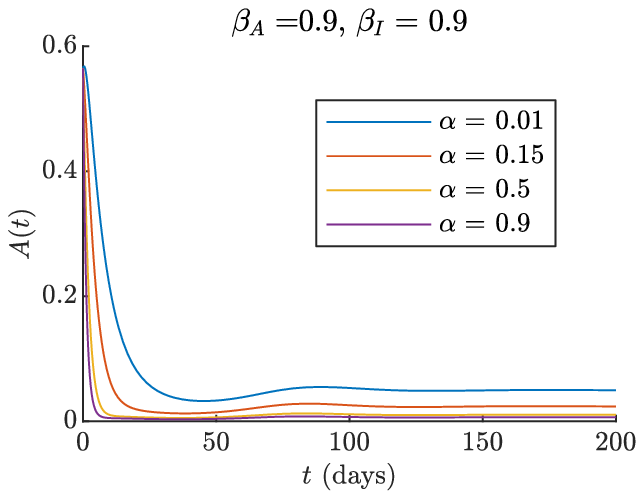}
        \caption{}\label{fig:at99}
    \end{subfigure}
    \begin{subfigure}{0.49\textwidth}
        \centering
        \includegraphics[width=0.9\textwidth]{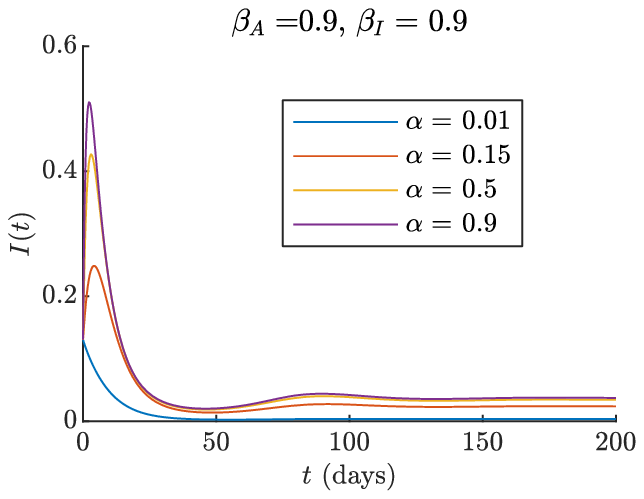}
        \caption{}\label{fig:it99}
    \end{subfigure}\hfill
    \begin{subfigure}{0.49\textwidth}
        \centering
        \includegraphics[width=0.9\textwidth]{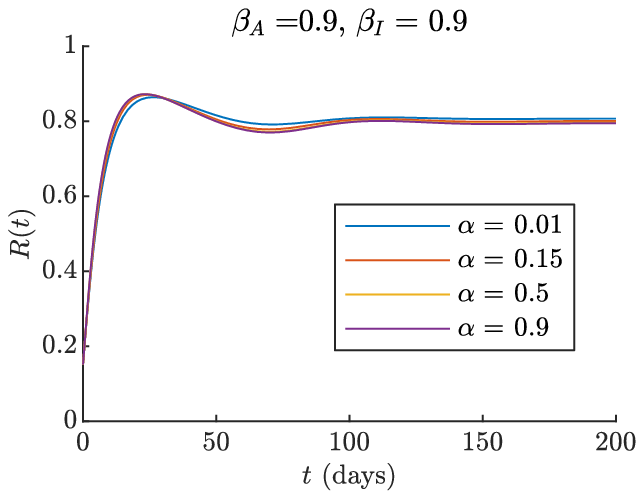}
        \caption{}\label{fig:rt99}
    \end{subfigure}
    \caption{Behavior of system (\ref{sairs_s}) as $\alpha$, the rate of symptoms onset, varies. Values of the parameters: $\mu=1/(70\cdot365)$, meaning an average lifespan of 70 years; $\beta_A =\beta_I =0.9$, $\nu=0.01$, $\gamma=1/100$, meaning the immunity lasts on average 100 days; $\alpha$ varying as shown, $\delta_A=0.125$, $\delta_I=0.15$.}
    \label{fig:talpha99}
\end{figure}

\begin{figure}[H]
    \centering
    \begin{subfigure}{0.49\textwidth}
        \centering
        \includegraphics[width=0.9\textwidth]{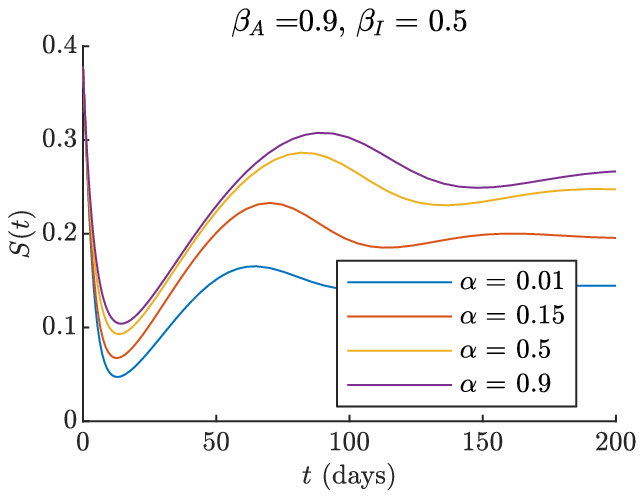}
        \caption{}\label{fig:st95}
    \end{subfigure}\hfill
    \begin{subfigure}{0.49\textwidth}
        \centering
        \includegraphics[width=0.9\textwidth]{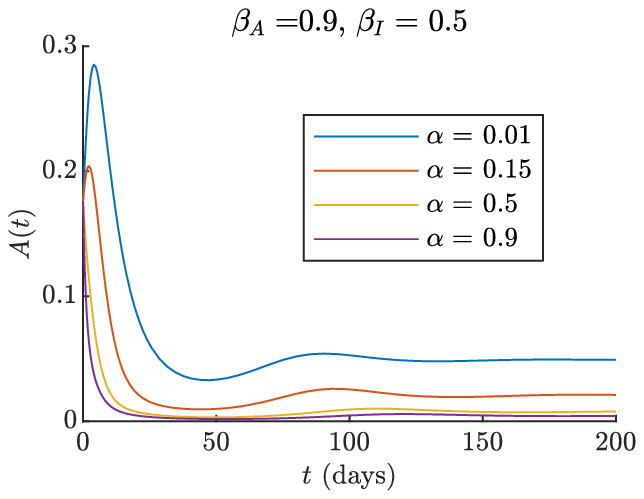}
        \caption{}\label{fig:at95}
    \end{subfigure}
    \begin{subfigure}{0.49\textwidth}
        \centering
        \includegraphics[width=0.9\textwidth]{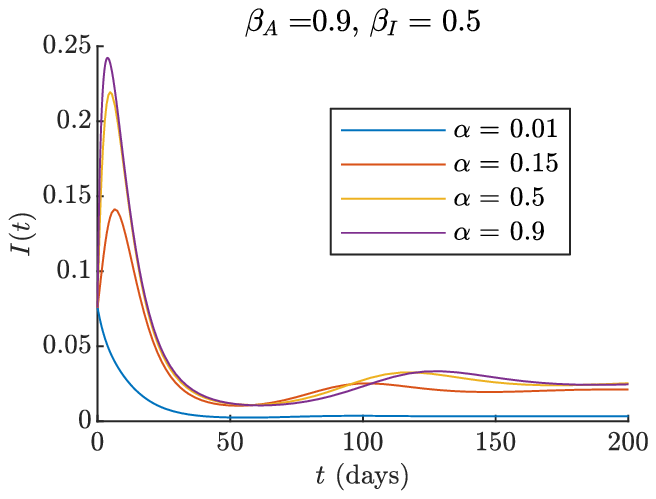}
        \caption{}\label{fig:it95}
    \end{subfigure}\hfill
    \begin{subfigure}{0.49\textwidth}
        \centering
        \includegraphics[width=0.9\textwidth]{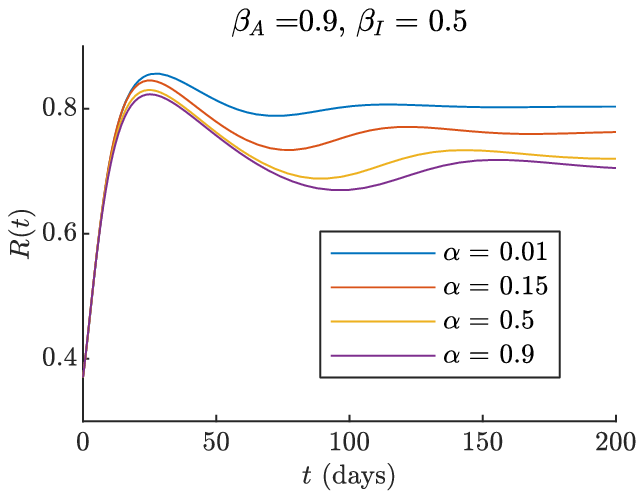}
        \caption{}\label{fig:rt95}
    \end{subfigure}
    \caption{Behavior of system (\ref{sairs_s}) as $\alpha$, the rate of symptoms onset, varies. Values of the parameters: $\mu=1/(70\cdot365)$, meaning an average lifespan of 70 years; $\beta_A =0.9$ $\beta_I =0.5$, $\nu=0.01$, $\gamma=1/100$, meaning the immunity lasts on average 100 days; $\alpha$ varying as shown, $\delta_A=0.125$, $\delta_I=0.15$.}
    \label{fig:talpha95}
\end{figure}

\begin{figure}[H]
    \centering
    \begin{subfigure}{0.49\textwidth}
        \centering
        \includegraphics[width=0.9\textwidth]{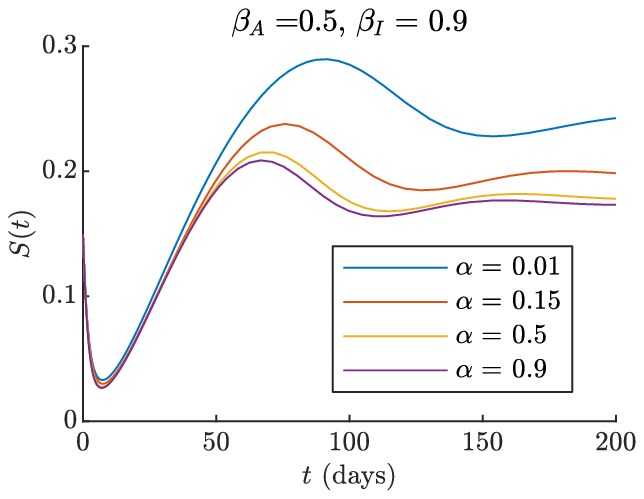}
        \caption{}\label{fig:st59}
    \end{subfigure}\hfill
    \begin{subfigure}{0.49\textwidth}
        \centering
        \includegraphics[width=0.9\textwidth]{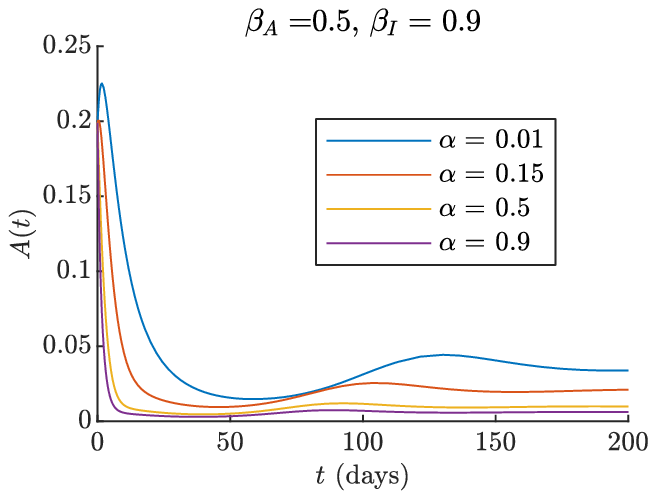}
        \caption{}\label{fig:at59}
    \end{subfigure}
    \begin{subfigure}{0.49\textwidth}
        \centering
        \includegraphics[width=0.9\textwidth]{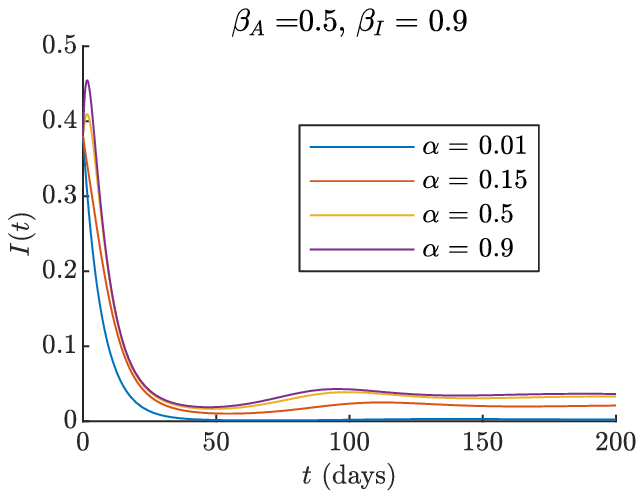}
        \caption{}\label{fig:it59}
    \end{subfigure}\hfill
    \begin{subfigure}{0.49\textwidth}
        \centering
        \includegraphics[width=0.9\textwidth]{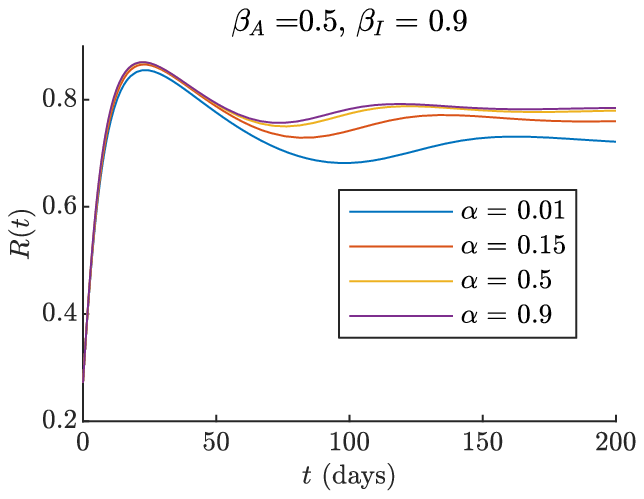}
        \caption{}\label{fig:rt59}
    \end{subfigure}
    \caption{Behavior of system (\ref{sairs_s}) as $\alpha$, the rate of symptoms onset, varies. Values of the parameters: $\mu=1/(70\cdot365)$, meaning an average lifespan of 70 years; $\beta_A =0.5$ $\beta_I =0.9$, $\nu=0.01$, $\gamma=1/100$, meaning the immunity lasts on average 100 days; $\alpha$ varying as shown, $\delta_A=0.125$, $\delta_I=0.15$.}
    \label{fig:talpha59}
\end{figure}

\begin{figure}[H]
    \centering
    \begin{subfigure}{0.49\textwidth}
        \centering
        \includegraphics[width=0.9\textwidth]{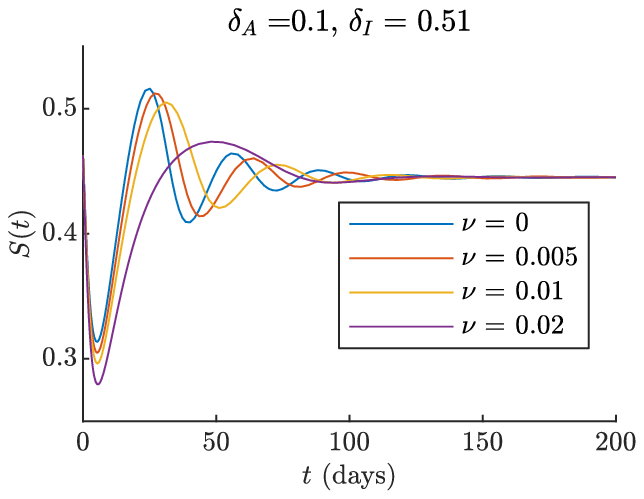}
        \caption{}\label{fig:stnu}
    \end{subfigure}\hfill
    \begin{subfigure}{0.49\textwidth}
        \centering
        \includegraphics[width=0.9\textwidth]{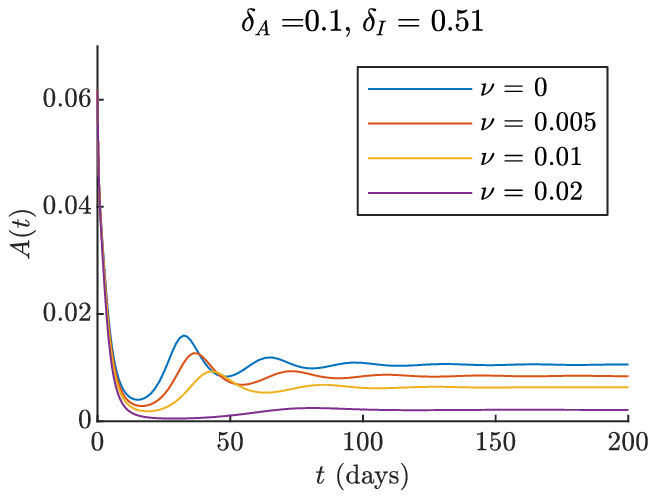}
        \caption{}\label{fig:atnu}
    \end{subfigure}
    \begin{subfigure}{0.49\textwidth}
        \centering
        \includegraphics[width=0.9\textwidth]{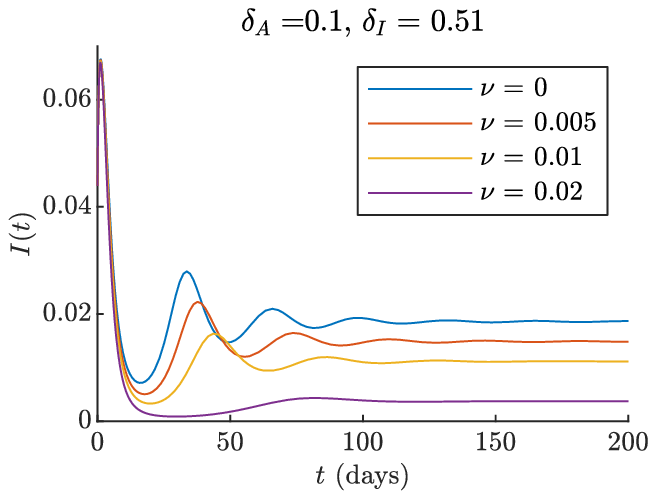}
        \caption{}\label{fig:itnu}
    \end{subfigure}\hfill
    \begin{subfigure}{0.49\textwidth}
        \centering
        \includegraphics[width=0.9\textwidth]{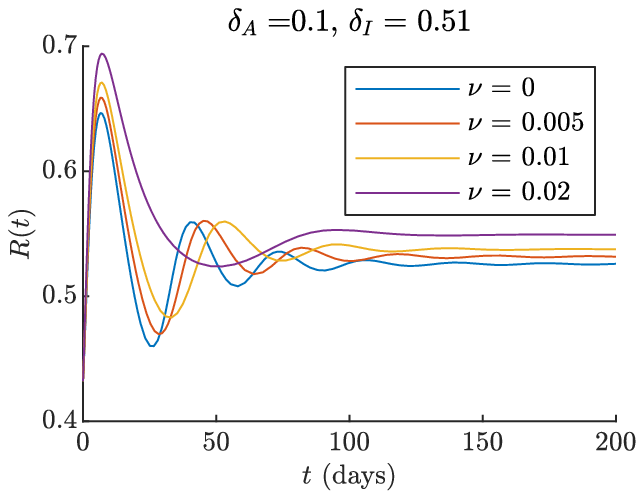}
        \caption{}\label{fig:rtnu}
    \end{subfigure}
    \caption{Behavior of system (\ref{sairs_s}) as $\nu$, the vaccination rate, varies. Values of the parameters: $\mu=1/(70\cdot365)$, meaning an average lifespan of 70 years; $\beta_A =0.5$ $\beta_I =0.9$, $\nu$ varying as shown, $\gamma=1/50$, meaning the immunity lasts on average 50 days; $\alpha=0.9$, $\delta_A=0.1$, $\delta_I=0.51$. The condition $\beta_A<\delta_I$ is satisfied.}
    \label{fig:tnucond}
\end{figure}

\section{Conclusions}\label{conclusion_sec}
We analyzed the behavior of an SAIRS compartmental model with vaccination. We determined the value of the basic reproduction number $\mathcal{R}_0$; then, we proved that the disease-free equilibrium is globally asymptotically stable, i.e..~the disease {eventually} dies out, if $\mathcal{R}_0<1$. Moreover, in the SAIRS-type model without vaccination ($\nu = 0$), we were able to generalize the result on the global asymptotic stability of the DFE also in the case $\mathcal{R}_0=1$.

Furthermore, we proved the uniform persistence of the disease and the existence of a unique endemic equilibrium if $\mathcal{R}_0>1$. Later, we analyzed the stability of this endemic equilibrium for some subcases of the model. 

The first case describes a disease which confers permanent immunity, i.e. $\gamma = 0$: the model reduces to an SAIR. In this framework, we answered the open problem presented in \cite{ansumali2020modelling}, including the additional complexity of vaccination: we proved the global asymptotic stability of the endemic equilibrium when $\mathcal{R}_0>1$.

We then proceeded to extend the results provided in \cite{robinson2013model} on the local stability analysis {for a SAIRS-type model}. We first considered the SAIRS model with the assumption that both asymptomatic and symptomatic infectious have the same transmission rate and recovery rate, i.e. $\beta_A= \beta_I$ and $\delta_A=\delta_I$, respectively. We were able to show that the endemic equilibrium is globally asymptotically stable if $\mathcal{R}_0>1$. Moreover, we analyzed the model without restrictions; we used the geometric approach proposed in \cite{lu2017geometric} to find the conditions under which the endemic equilibrium is globally asymptotically stable. We proved the global stability in the case $\mathcal{R}_0>1$ and $\beta_A < \delta_I$. 

{We leave, as an open problem, the global asymptotic stability of the endemic equilibrium without any restriction on the parameters}: we conjecture that the global asymptotic stability for the endemic equilibrium only requires $\mathcal{R}_0>1$, as our numerical simulations suggest. 

{Many generalizations and investigations of our model are possible.
For example,} we considered the vital dynamics without distinguish between natural death and disease related deaths; an interesting, although complex, generalization of our model could explore the implications of including disease-induced mortality.

A natural extension of our SAIRS model could take into account different groups of individual among which an epidemic can spread. One modelling approach for this are multi-group compartmental models. Other more realistic extensions may involve a greater number of compartments, for example the ``Exposed" group, or time-dependent parameters which can describe the seasonality of a disease or some response measures from the population, as well as non-pharmaceutical interventions.

\section*{Acknowledgments}
The authors would like to thank Prof.~Andrea Pugliese and Prof.~Bruno Buonomo for the fruitful discussions and suggestions during the writing of this paper. The authors thank also Prof.~Marco Broccardo for discussions and a careful reading of the paper draft.\\

The research of Stefania Ottaviano was supported by the University of Trento in the frame ``SBI-COVID - Squashing the business interruption curve while flattening pandemic curve (grant 40900013)''. The research of Mattia Sensi was supported by the TUDelft project ``Epidemics over Human Contact Graphs''.

\bibliographystyle{ieeetr}
\bibliography{biblio1}

\end{document}